\documentclass{amsart}

\usepackage{amssymb}
\usepackage{amsthm}
\usepackage{amsmath}
\usepackage{mathrsfs}			
\usepackage[usenames,dvipsnames]{color}
\usepackage[unicode]{hyperref}	
\hypersetup{urlcolor=blue, colorlinks=true}		
\usepackage{geometry}
\usepackage{subfigure}
\usepackage[bf]{caption}
\input xy
\xyoption{all}
\usepackage{tikz}
\usepackage{tikz-cd}

\renewcommand{\thesubfigure}{(\alph{subfigure})}

\newcommand{\QQ}{\mathbf{Q}}
\newcommand{\ZZ}{\mathbf{Z}}
\newcommand{\RR}{\mathbf{R}}
\newcommand{\CC}{\mathbf{C}}

\newcommand{\ds}{\displaystyle}

\newcommand{\ul}[1]{\underline{#1}}
\newcommand{\mf}[1]{\mathfrak{#1}}
\newcommand{\mc}[1]{\mathcal{#1}}

\newcommand{\gl}{\mathrm{GL}}

\renewcommand{\mod}[1]{\text{ }(\operatorname{mod}\text{ }#1)}
\renewcommand{\varepsilon}{\epsilon}

\DeclareMathOperator{\sgn}{sgn}
\DeclareMathOperator{\Gal}{Gal}
\DeclareMathOperator{\disc}{disc}
\DeclareMathOperator{\tr}{tr}

\newcommand{\vect}[1]{\begin{pmatrix}#1\end{pmatrix}}

\newtheorem{theorem}{Theorem}[section]
\newtheorem{theoremA}{Theorem}
\newtheorem{proposition}[theorem]{Proposition}
\newtheorem{lemma}[theorem]{Lemma}
\newtheorem{corollary}[theorem]{Corollary}

\theoremstyle{definition}
\newtheorem{remark}[theorem]{Remark}
\newtheorem{definition}[theorem]{Definition}

\numberwithin{equation}{section}

\hypersetup{pdfauthor={Piper H\ and Robert Harron},pdftitle={The shapes of Galois quartic fields},pdfkeywords={Algebraic number theory, quartic fields, lattices, arithmetic statistics}}

\newcommand{\conorm}[7]{!{(-100,0)}*{p}}

\newcommand{\nr}{{\mathrm{nr}}}
\newcommand{\sh}{\mathrm{sh}}
\newcommand{\lr}{\langle\,\cdot\,,\,\cdot\rangle}
\newcommand{\ga}{{\gcd}^\ast}

\renewcommand{\L}{{\mathcal{L}}}
\renewcommand{\C}{{\mathcal{C}}}
\newcommand{\W}{{\mathcal{W}}}

\renewcommand{\sf}{{\mathrm{sf}}}
\newcommand{\tame}{{\mathrm{tame}}}
\newcommand{\wild}{{\mathrm{wild}}}
\newcommand{\rrat}{{\mathrm{rrat}}}

\begin{document}

\title{T\MakeLowercase{he shapes of }G\MakeLowercase{alois quartic fields}}

\author{Piper H}
\address{
Department of Mathematics\\
Keller Hall\\
University of Hawai`i at M\={a}noa\\
Honolulu, HI 96822\\
USA
}
\email{piper@math.hawaii.edu}

\author{Robert Harron}
\address{
Department of Mathematics\\
Keller Hall\\
University of Hawai`i at M\={a}noa\\
Honolulu, HI 96822\\
USA
}
\email{rharron@math.hawaii.edu}
\thanks{The second author is partially supported by a Simons Collaboration Grant.}

\begin{abstract}
We determine the shapes of all degree $4$ number fields that are Galois. These lie in four infinite families depending on the Galois group and the tame versus wild ramification of the field. In the $V_4$ case, each family is a two-dimensional space of orthorhombic lattices and we show that the shapes are equidistributed, in a regularized sense, in these spaces as the discriminant goes to infinity (with respect to natural measures). We also show that the shape is a complete invariant in some natural families of $V_4$-quartic fields. For $C_4$-quartic fields, each family is a one-dimensional space of tetragonal lattices and the shapes make up a discrete subset of points in these spaces. We prove asymptotics for the number of fields with a given shape in this case.
\end{abstract}


\maketitle

 
 \tableofcontents
 
\section{Introduction}

The shape of a number field $K$ of degree $n$ is an equivalence class of lattices of rank $n-1$ (up to rotations, reflections, and scaling) that arises from the geometry of numbers. The study of this invariant began with the PhD thesis of David Terr (\cite{Terr}) in which it is shown that the shapes of both real and complex cubic fields are equidistributed (as the discriminant goes to infinity) in the space of shapes of rank $2$ lattices (i.e.\ the upper-half plane modulo the action of $\gl_2(\ZZ)$ by fractional linear transformations). Manjul Bhargava and the first author generalized this result to $S_4$-quartic and $S_5$-quintic fields in \cite{Manjul-Piper,PiperThesis}, conjecturing that such a `random' behaviour should hold for degree $n$ $S_n$-number fields for all $n$. On the other hand, also in \cite{Terr}, Terr shows that all Galois cubic fields have the same shape: hexagonal! The argument there is quite simple: the order $3$ automorphism of a Galois cubic field yields an order $3$ automorphism of its shape and the hexagonal lattice is the only rank $2$ lattice with an automorphism of order 3. This kind of argument drastically loses its strength when moving on to Galois quartic fields: there are infinitely many rank $3$ lattices containing an order $4$ automorphism or three order $2$ automorphisms. In this article, we determine the shapes of all Galois quartic fields showing that they lie in four infinite families depending on whether the Galois group is $C_4$ or $V_4$ and whether the field is tamely ramified or wildly ramified. This kind of Tame-Wild dichotomy was pointed out by the second author in \cite{PureCubicShapes} and also arises in \cite{ComplexCubics}. We also investigate how the shapes are distributed in these four families. As in \cite{PureCubicShapes,ComplexCubics}, the distribution of shapes in the $V_4$ case provides an explanation for the occurrence of log terms in the asymptotics of counting the number fields in question. We go into more detail now, considering each Galois group separately.

\subsection{Statement of results: $V_4$ case}

We will show in \S\ref{sec:V4_shapes} that the shape of a $V_4$-quartic field $K$ is an orthorhombic lattice, meaning that it can be described using a right rectangular prism. The field $K$ is determined by its three quadratic subfields $\QQ(\sqrt{\Delta_i})$ (with $\Delta_i$ a fundamental discriminant) and its shape is described by saying the ratios of the lengths of the sides of the prism are $\sqrt{|\Delta_1|}:\sqrt{|\Delta_2|}:\sqrt{|\Delta_3|}$ (see Theorem~\ref{thm:V4_main_theorem} below for more details). As a consequence, we obtain the following theorem saying that, within certain natural families of $V_4$-quartic fields, the shape determines the field.

\begin{theoremA}[Corollary~\ref{cor:V4_shape_complete_invariant} below]\mbox{}
	\begin{enumerate}
		\item The shape of a totally real $V_4$-quartic field determines it amongst the family of all totally real $V_4$-quartic fields.
		\item The shape of a tamely ramified $V_4$-quartic field determines it amongst the family of all tamely ramified $V_4$-quartic fields.
	\end{enumerate}
\end{theoremA}

\begin{remark}\mbox{}
	\begin{enumerate}
		\item Note that the \textit{discriminant} of a totally real $V_4$-quartic field is not a complete invariant. For instance, $\QQ(\sqrt{10},\sqrt{13})$ and $\QQ(\sqrt{10},\sqrt{26})$ both have discriminant $2^6\cdot5^2\cdot13^2$, but are not isomorphic. Their shapes are however distinct.
		\item This result is complementary to a recent result of Carlos Revera-Guaca and Guillermo Mantilla-Soler that says e.g.\ that in the family of totally real quartic fields with fundamental discriminant the shape is a complete invariant \cite[Theorem~2.12]{RGMS}. The discriminants of $V_4$-quartic fields are never fundamental.
		\item If $D_1\equiv D_2\equiv2\mod{4}$ (with $D_i>0$ squarefree), then $\QQ(\sqrt{D_1},\sqrt{D_2})$ and $\QQ(\sqrt{-D_1},\sqrt{-D_2})$ have the same shape, as do $\QQ(\sqrt{-D_1},\sqrt{D_2})$ and $\QQ(\sqrt{D_1},\sqrt{-D_2})$.
		\item The second author has shown in \cite{ComplexCubics} that the shape of a complex cubic field determines that field within the family of all cubic fields. In that case, the shape is a two-dimensional lattice and, as such, is given by a point in the complex upper-half plane. The complex cubic field is then obtained by adjoining to $\QQ$ a coordinate of the shape. This is similar to what is happening here since we can think of $\sqrt{|\Delta_i|}$ as coordinates describing the shape of the $V_4$-quartic.
	\end{enumerate}
\end{remark}

We have the following further example of the Tame-Wild dichotomy extending what was pointed out in \cite{PureCubicShapes,ComplexCubics}. Note that a Galois quartic field is wildly ramified if and only if $2$ ramifies.
\begin{theoremA}\label{thm:V4TameWild}
	The shape of a $V_4$-quartic field $K$ lies in one of two spaces based upon whether $2$ ramifies. When $2$ ramifies in $K$, its shape lies in the family $\mc{S}_{oC}$ of base-centered orthorhombic lattices. Otherwise, the shape is in the family $\mc{S}_{oI}$ of body-centered orthorhombic lattices.
\end{theoremA}

Once we break up the fields according to being tame or wild, we can ask whether the shapes are ``random'' in the spaces $\mc{S}_{oC}$ and $\mc{S}_{oI}$, respectively (endowed with natural measures $\mu_{oC}$ and $\mu_{oI}$, respectively). We will show that this is the case. As in \cite{PureCubicShapes}, the spaces $\mc{S}_{oC}$ and $\mc{S}_{oI}$ have infinite measure and the asymptotics for counting these fields have log terms. We must therefore ``regularize'' our notion of equidistribution in a similar way. We prove the following result in \S\ref{sec:V4equid}.

\begin{theoremA}\label{thm:V4equid}
	The shapes of $V_4$-quartic fields are equidistributed, in a regularized sense, within the two-dimensional space in which they live. Specifically, let
	\[
		C_\wild=\frac{5}{48}\prod_{p\text{ odd}}\left(1-6p^{-2}+8p^{-3}-3p^{-4}\right)
	\]
	and
	\[
		C_\tame=\frac{1}{6}\prod_{p\text{ odd}}\left(1-6p^{-2}+8p^{-3}-3p^{-4}\right),
	\]
	where each infinite product is over all odd primes. If $W$ is a compact $\mu_{oC}$- or $\mu_{oI}$-continuity set,\footnote{Recall that a $\mu$-\textit{continuity set} for a measure $\mu$ is a measurable set whose boundary has measure $0$.} respectively, then
	\[
		\lim_{X\rightarrow\infty}\dfrac{N_\wild(X,W)}{X^{1/2}}=C_\wild\mu_{oC}(W)
	\]
	and
	\[
		\lim_{X\rightarrow\infty}\dfrac{N_\tame(X,W)}{X^{1/2}}=C_\tame\mu_{oI}(W),
	\]
	where $N_\wild(X,W)$ (resp.\ $N_\tame(X,W)$) denotes the number of $V_4$-quartic fields with discriminant bounded by $X$, shape in $W$, that are wildly (resp.\ tamely) ramified.
\end{theoremA}

\begin{remark}\mbox{}
	\begin{enumerate}
		\item Andrew Baily showed in \cite{Baily} that the number of $V_4$-quartic fields with discriminant bounded by $X$ grows like
		\[
			X^{1/2}\log^2(X).
		\]
		The usual notion of equidistribution would have the denominator in the limits above be this $X^{1/2}\log^2(X)$. We say ``in a regularized sense'' to indicate that we have modified this denominator.
		As in \cite{PureCubicShapes}, we show that requiring the fields to have shape in some \textit{compact} set removes the log factors. Furthermore, we can ``see'' both of the log factors in the space of shapes: if $W$ is a ``box'' constraining the (two) shape parameters to lie between $1$ and $R$, then our results shows that the number of fields grows like
		\[
			X^{1/2}\log^2(R).
		\]
		This seems to indicate that we might be able to better understand log terms in the asymptotics of counts of number fields if we understand the spaces in which their shapes live. We refer to \cite[\S1.3]{ComplexCubics} for further discussion of this phenomenon.
		\item We may phrase this result in terms of weak convergence of measures as in \cite[\S3.1]{PureCubicShapes}.
		\item We prove these results by parametrizing the $V_4$-quartic fields in question using strongly carefree triples satisfying certain congruence conditions and lying in some region. We use the Principle of Lipschitz and a sieve to count these triples.
		\item Our results on the determination and equidistribution of shapes of $V_4$-quartic fields is generalized to totally real tame $C_2^3$-octic fields in the PhD thesis of Jamal Hassan.
	\end{enumerate}
\end{remark}

\subsection{Statement of results: $C_4$ case}
In \S\ref{sec:C4fields_description}, we show that the shapes of $C_4$-quartic fields are tetragonal lattices, i.e.\ they can be described by a right rectangular prism with square base. The ratio of the height to the side length of the base, which we call the \textit{side ratio}, is given by an interesting ramification invariant, as follows. A $C_4$-quartic field $K$ has a unique quadratic subfield $K_2$. Let $\Delta_2$ denote its discriminant and let $\mc{N}$ denote the absolute norm of the relative discriminant of $K/K_2$. Let
\[
	\rrat_K:=\frac{\mc{N}}{|\Delta_2|},
\]
which we will call the \textit{ramification ratio} of $K$. Note that every prime that ramifies in $K_2$ must ramify in $K$,\footnote{Indeed, the inertia field of a prime ramified in $K$ can only be $K_2$ or $\QQ$, so that if it ramifies in $K_2$, its ramification index must be $4$.} so that $\rrat_K$ is a positive integer. In fact, $\rrat_K=(2^eA)^2$ where $0\leq e\leq3$ and $A$ can be any squarefree, odd integer ($A$ is the product of the odd primes that ramify in $K$, but not in $K_2$). The ramification ratio of $K$ dictates the shape of $K$ as given in the following theorem (see Theorem~\ref{thm:C4_main_theorem} for a more precise statement).
\begin{theoremA}
	The shape of a $C_4$-quartic field $K$ lies in one of two families depending on whether $2$ ramifies. When $2$ ramifies in $K$, the shape is a primitive tetragonal lattice with side ratio $\sqrt{2}\cdot\rrat_K^{-1/4}$. Otherwise, the shape is a body-centered tetragonal lattice with side ratio $\rrat^{-1/4}$.
\end{theoremA}
\begin{remark}\mbox{}
\begin{enumerate}
\item Since the discriminant of a $C_4$-quartic field is of the form $2^fA^2D^3$, where $D$ is squarefree and relatively prime to $A$, and $\rrat_K=(2^eA)^2$ and $e$ is determined by $f$, the discriminant of a $C_4$ field determines its shape.
	\item In a recent preprint, Wilmar Bola\~{n}os and Mantilla-Soler compute a Gram matrix for the trace form of any tame cyclic field (of arbitrary degree). In particular, a Gram matrix representing the shape of a totally real tame $C_4$-quartic field can be obtained from \cite[Corollary~3.11]{BMS}.
\end{enumerate}
\end{remark}

Since $\rrat_K$ is an integer, these shapes yield discrete sets of points in the spaces of tetragonal lattices. As such they are not dense, let alone equidistributed. On the other hand, we are able to count how many fields have a given shape. There are infinitely many fields with a given shape so we provide asymptotics for the number of such fields with bounded discriminant.
\begin{theoremA}
	Let $A$ be a squarefree, odd integer and for primes $p$, let
	\[
		f_A(p)=\begin{cases}
					2&\text{if }p\equiv1\mod{4}\text{ and }p\nmid A,\\
					0&\text{otherwise.}
				\end{cases}
	\]
	Let
	\[
		C_{\Sigma_A}:=\prod_{p\text{ prime}}\left(1-\frac{f_A(p)}{p}\right)\left(1-\frac{1}{p}\right).
	\]
	\begin{enumerate}
		\item Let $N_\tame(X;A)$ be the number of tamely ramified $C_4$-quartic fields $K$ with $\rrat_K=A^2$, ${\Delta_K\leq X}$, and that are totally real or totally imaginary according whether $A>0$ or not. Then,
		\[
			N_\tame(X;A)=\frac{C_{\Sigma_A}}{2^2\cdot\rrat_K^{1/3}}X^{1/3}+o(X^{1/3}).
		\]
		\item Let $N_\wild(X;A,e)$ (for $e=1,2$, or $3$) be the number of wildly ramified $C_4$-quartic fields $K$ with $\rrat_K=(2^eA)^2$, ${\Delta_K\leq X}$, and that are totally real or totally imaginary according whether $A>0$ or not. Then,
		\[
			N_\wild(X;A,e)=\frac{C_{\Sigma_A}}{2^{2^{3-e}}\cdot\rrat_K^{1/3}}X^{1/3}+o(X^{1/3}).
		\]
	\end{enumerate}
\end{theoremA}

\penalty-100
\begin{remark}\mbox{}
	\begin{enumerate}
		\item This result is proved in \S\ref{sec:C4shapes_distribution} by relating the counts to how many ways certain integers can be written as a sum of two squares. The tools we use are the Wirsing--Odoni method and the Wiener--Ikehara Tauberian Theorem.
		\item In Theorem~\ref{thm:C4_asymptotics}, for $N_\wild(X;A,e)$ for $e=1,3$, we in fact get an error of $O\!\left(X^{1/3}/(\log X)^{1-\epsilon}\right)$, for all $\epsilon>0$.
		\item The proportions of the number of fields with different $\rrat_K$ depends arithmetically on the value of $\rrat_K$. For instance, if $p$ is an odd prime, the proportion of fields with ramification ratio $\rrat_K$ versus those with ramification ratio $p\cdot\rrat_K$ is
		\[
		\begin{cases}
			p^{2/3}+\dfrac{2}{p^{1/3}}	& p\equiv1\mod{4}\\
			p^{2/3}				& p\equiv3\mod{4}.
		\end{cases}
	\]
	This implies that these proportions do not arise simply from the action of some real Lie group as is the case for the measure in the $V_4$ case (see Remark~\ref{rem:invariantmeasure}).
	\end{enumerate}
\end{remark}

\subsection{Outline of this article}
In \S\ref{sec:shape}, we recall some basic definitions and facts concerning shapes of number fields. In \S\ref{sec:lattices}, we overview the relevant features of conorm diagrams, as introduced by John Conway and Neil Sloane \cite{Conway-Sloane}. This provides an elegant and convenient way to treat rank $3$ lattices. This section also contains a discussion of the natural measures that come with the spaces of orthorhombic lattices we study. We then move on to proving the main results of this article in the remaining four sections, beginning with the $V_4$ case, then the $C_4$ case. In sections \ref{sec:V4_shapes} and \ref{sec:C4fields_description}, we determine the shapes of $V_4$- and $C_4$-quartic fields, respectively. The equidistribution of shapes of $V_4$-quartics is shown in \S\ref{sec:V4equid} and the asymptotics for $C_4$-quartics of a given shape are derived in \S\ref{sec:C4shapes_distribution}.

\section{The shape of a number field}\label{sec:shape}
In this brief section, we recall the notion of the shape of a number field. For additional number-theoretic details, we refer the reader to \cite[\S I.5]{Neukirch}.

The \emph{shape} of a rank $d$ lattice $\Lambda$ in a real inner product space $V$ is its equivalence class under orthogonal transformations and homotheties. The shape can be encoded as a Gram matrix modulo a change-of-basis action by $\gl_d(\ZZ)$ and a scaling action by $\RR^\times$, as follows. Given a basis $B=(v_1,\dots,v_d)$ of $\Lambda$, we may form its Gram matrix $G_B:=(\langle v_i, v_j\rangle)$, where $\lr$ denotes the inner product on $V$. If $B^\prime=(v_1^\prime,,\dots,v_d^\prime)$ is another basis of $\Lambda$, then there is an element $g\in\gl_n(\ZZ)$ such that
\[
	\vect{v_1^\prime\\\vdots\\v_d^\prime}=g\vect{v_1\\\vdots\\v_d}
\]
(and vice versa). The bilinearity of $\lr$ implies that
\[
	G_{B^\prime}=gG_Bg^T.
\]
Accordingly, letting $\mc{G}$ denote the space of positive definite symmetric $d\times d$ real matrices, we define a left action of $g\in\gl_2(\RR)$ on $G\in\mc{G}$ by
\[
	g\cdot G:=gGg^T.
\]
We can also think of $\RR^\times$ as acting on the basis $B$ by scaling. The (right) action of $\lambda\in\RR^\times$ on $G\in\mc{G}$ is then $G\cdot\lambda:=\lambda^2G$. We then have a bijection between shapes of rank $d$ lattices (i.e.\ lattices up to rotations, reflections, and scaling) and the set
\[
	\gl_d(\ZZ)\backslash\mc{G}/\RR^\times,
\]
the map being given by taking $B$ to be a basis of $\Lambda$ and sending $\Lambda$ to $\sh(\Lambda):=\gl_d(\ZZ)\cdot G_B\cdot\RR^\times$.\footnote{The surjectivity of this map can be seen as a consequence of the spectral theorem for symmetric real matrices.}

The geometry of numbers attaches a lattice to a number field $K$, as follows. Let $K$ be a number field of degree $n$ and let $\sigma_1,\dots,\sigma_n:K\rightarrow\CC$ denote its $n$ complex embeddings. We call the map $j:K\rightarrow\CC^n$ given by $\alpha\mapsto(\sigma_1(\alpha),\dots,\sigma_n(\alpha))$ the \textit{Minkowski embedding of $K$}. It is a fundamental result of the geometry of numbers that the $\RR$-span of the image of $j$ is an $n$-dimensional real inner product space (where the inner product is the restriction of the standard Hermitian inner product on $\CC^n$). We call this space the \textit{Minkowski space of $K$} and denote it by $K_\RR$. The image of the ring of integers $\mc{O}_K$ of $K$ under $j$ is a lattice of rank $n$ in $K_\RR$ that we denote $\Lambda_K$. The covolume of this lattice is $\sqrt{|\Delta_K|}$ so that as $|\Delta_K|\rightarrow\infty$ the $\Lambda_K$ get ``bigger''. However, the vector $j(1)$ is of constant length $\sqrt{n}$, thus skewing the shapes of the $\Lambda_K$ in a family of degree $n$ fields ordered by discriminant. We therefore define the \textit{shape of $K$}, denoted $\sh(K)$, to be the shape of the lattice $\Lambda_K^\perp$ obtained by taking the orthogonal projection of $\Lambda_K$ onto the orthogonal complement of $j(1)$.

Concretely, we will frequently obtain the shape as follows. First note, that for any $\alpha\in K$,
\[
	\langle j(1),j(\alpha)\rangle=\tr(\alpha),
\]
where $\tr:K\rightarrow\QQ$ is the usually trace map of the field extension $K/\QQ$. We may therefore define a ``perp map'' from $K$ to itself by\footnote{We have scaled by $n$ so as to preserve integrality, i.e.\ so that the image of $\mc{O}_K$ under the perp map lies in $\mc{O}_K$. This is not strictly necessary for our purposes, but is convenient.}
\[
	\alpha^\perp:=n\alpha-\tr(\alpha).
\]
Letting $\mc{O}_K^\perp$ denote the image of $\mc{O}_K$ under the perp map, we then get, from standard linear algebra formulas for orthogonal projection, that $j(\mc{O}_K^\perp)=n\Lambda_K^\perp$. Therefore, the shape of $K$ is also the shape of the lattice $j(\mc{O}_K^\perp)$. If $(1,\gamma_1,\dots,\gamma_{n-1})$ is an integral basis of $K$, then $(n\gamma_1-\tr(\gamma_1),\dots,n\gamma_{n-1}-\tr(\gamma_{n-1}))$ is a $\ZZ$-basis of $\mc{O}_K^\perp$. With this, we can explicitly calculate the shape of $K$ knowing such an integral basis.

\section{Preliminaries on rank 3 lattices}\label{sec:lattices}
This section recalls an elegant theory due to Conway and Sloane (\cite{Conway-Sloane}) for parametrizing rank $3$ lattices. The so-called conorm diagrams of rank $3$ matrices are quite close to Gram matrices, but understanding when two of them correspond to the same lattice is simpler. Conorm diagrams also allow for an easy determination of the Voronoi cell of a lattice. After a brief overview of the theory of conorm diagrams (following \cite{Conway-Sloane}), we produce the conorm diagrams for the families of lattices that arise in our study of shapes of Galois quartic fields. We end this section by defining natural measures on spaces of orthorhombic lattices for use in our theorem on the equidistribution of shapes of $V_4$-quartic fields.

\subsection{Voronoi reduction theory}

We refer the reader to \cite{Conway-Sloane} for more details.

The term \emph{putative conorm diagram} refers to a labeling of the points of the Fano plane (or, really, its dual) $\mathbf{P}^2(\mathbf{F}_2)$ by real numbers. Here is why. Let $\Lambda$ be a rank 3 lattice in a Euclidean space. An \emph{obtuse superbase} of $\Lambda$ is a quadruple $(v_0,v_1,v_2,v_3)$ of vectors in $\Lambda$ such that
	\begin{itemize}
		\item $(v_1,v_2,v_3)$ is a basis of $\Lambda$,
		\item $v_0+v_1+v_2+v_3=0$, and
		\item $v_i\cdot v_j\leq0$ for all $i\neq j$ (the \emph{obtuse} condition).
	\end{itemize}
	A quadruple $(v_0,v_1,v_2,v_3)$ that does not necessarily satisfy the third condition is simply called a \emph{subperbase}.
Given a superbase, let $-p_{ij}=v_i\cdot v_j$ for $i\neq j$; these are the \emph{putative conorms} of $\Lambda$. These numbers are encoded on the Fano plane (or, really, its dual) as in Figure~\ref{fig:conorm_diagram_definition} in what is called the \emph{putative conorm diagram} of the superbase.
\begin{figure}[h]
	\begin{center}
		\subfigure{
			\begin{tikzpicture}[scale=1.8]
				\draw [-] (-1,0) -- (1,0);
				\draw [-] (-1,0) -- (0,-1.732);
				\draw [-] (1,0) -- (0,-1.732);
				\draw [-] (-1,0) -- (0.5,-0.866);
				\draw [-] (1,0) -- (-0.5,-0.866);
				\draw [-] (0,0) -- (0,-1.732);
				\draw[black] (0,-0.577) circle [x radius=0.577, y radius=0.577];
				\draw[black, fill] (0,0) circle [x radius=0.03, y radius=0.03];
				\draw[black, fill] (1,0) circle [x radius=0.03, y radius=0.03];
				\draw[black, fill] (-1,0) circle [x radius=0.03, y radius=0.03];
				\draw[black, fill] (0,-1.732) circle [x radius=0.03, y radius=0.03];
				\draw[black, fill] (-0.5,-0.866) circle [x radius=0.03, y radius=0.03];
				\draw[black, fill] (0.5,-0.866) circle [x radius=0.03, y radius=0.03];
				\draw[black, fill] (0,-0.577) circle [x radius=0.03, y radius=0.03];
				
				\node at (0.08,-0.42) {$0$};
				\node at (0.7, -0.9) {$p_{12}$};
				\node at (0, 0.15) {$p_{13}$};
				\node at (-0.7, -0.9) {$p_{23}$};
				\node at (1.15, 0.1) {$p_{01}$};
				\node at (0, -1.87) {$p_{02}$};
				\node at (-1.15, 0.1) {$p_{03}$};
			\end{tikzpicture}
			}
	\caption{\label{fig:conorm_diagram_definition}Conorm diagram of an obtuse superbase.}
	\end{center}
\end{figure}
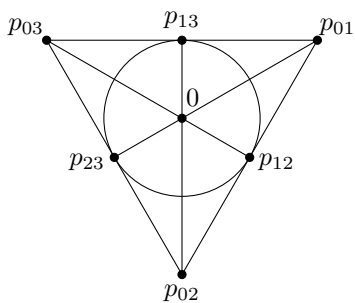
If the superbase is in fact obtuse, one removes the word putative everywhere. In other words, a \emph{conorm diagram} is a putative conorm diagram of an obtuse superbase of some $\Lambda$ (really, up to some automorphism of the Fano plane). The main theorem of \cite{Conway-Sloane} says
\begin{itemize}
	\item the collection of conorm diagrams is exactly those putative conorm diagrams whose entries are non-negative with minimum 0 and whose support does not lie in a proper subspace;
	\item every rank 3 lattice has an obtuse superbase;
	\item	two lattices are isomorphic if and only if their conorm diagrams differ by an automorphism of the Fano plane.
\end{itemize}

Conway and Sloane develop an algorithm they call \emph{Voronoi reduction} which transforms a putative conorm diagram for $\Lambda$ into a conorm diagram for $\Lambda$.

We note that the above results show that two lattices have the same shape if and only if there is an automorphism of the Fano plane that brings one conorm diagram to a scaled version of the other.

\subsection{Combinatorial type of a lattice}

Recall that the Voronoi cell of a lattice is the set of points closer to the origin than to any other lattice point.

\begin{theorem}[\cite{Fedorov85,Fedorov91}, {\cite[Figure~7 and Theorem~9]{Conway-Sloane}}]
	The Voronoi cells of rank 3 lattices come in 5 combinatorially distinct\footnote{By \emph{combinatorially distinct}, we mean that the triples $(V,E,F)$ encoding the number of vertices, edges, and faces, are distinct.} families represented by the 5 primary parallelohedra:
	\begin{itemize}
		\item[(I)] the \emph{truncated octahedron}, with $(V,E,F)=(24,36,14)$;
		\item[(II)] the \emph{rhombo-hexagonal dodecahedron}, with $(V,E,F)=(18,28,12)$;
		\item[(III)] the \emph{rhombic dodecahedron}, with $(V,E,F)=(14,24,12)$;
		\item[(IV)] the \emph{hexagonal prism}, with $(V,E,F)=(12,18,8)$;
		\item[(V)] the \emph{cuboid}, with $(V,E,F)=(8,12,6)$.
	\end{itemize}
	The family in question can be read off from the configuration of zeroes in the conorm diagram. See Figure~\ref{fig:conorm_diagrams} for the general conorm diagrams of each family.
\end{theorem}
\renewcommand{\thesubfigure}{\,\,(\Roman{subfigure})}
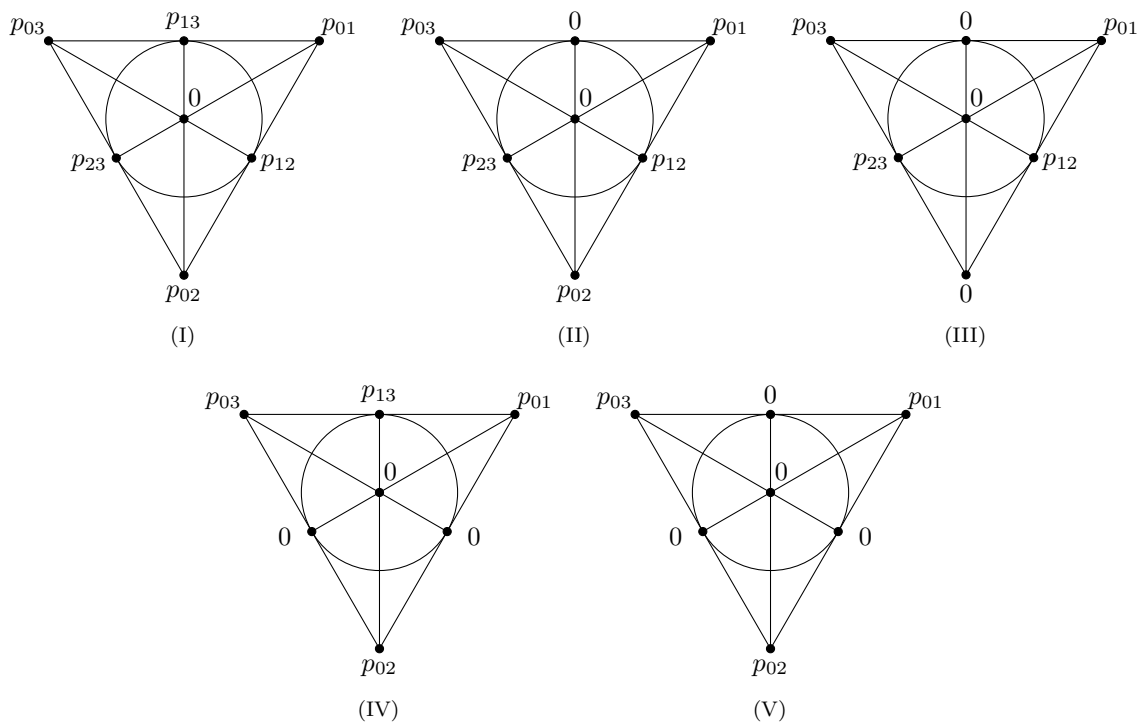
\begin{figure}[h]
	\begin{center}
		\subfigure[]{
			\begin{tikzpicture}[scale=1.8]
				\draw [-] (-1,0) -- (1,0);
				\draw [-] (-1,0) -- (0,-1.732);
				\draw [-] (1,0) -- (0,-1.732);
				\draw [-] (-1,0) -- (0.5,-0.866);
				\draw [-] (1,0) -- (-0.5,-0.866);
				\draw [-] (0,0) -- (0,-1.732);
				\draw[black] (0,-0.577) circle [x radius=0.577, y radius=0.577];
				\draw[black, fill] (0,0) circle [x radius=0.03, y radius=0.03];
				\draw[black, fill] (1,0) circle [x radius=0.03, y radius=0.03];
				\draw[black, fill] (-1,0) circle [x radius=0.03, y radius=0.03];
				\draw[black, fill] (0,-1.732) circle [x radius=0.03, y radius=0.03];
				\draw[black, fill] (-0.5,-0.866) circle [x radius=0.03, y radius=0.03];
				\draw[black, fill] (0.5,-0.866) circle [x radius=0.03, y radius=0.03];
				\draw[black, fill] (0,-0.577) circle [x radius=0.03, y radius=0.03];
				
				\node at (0.08,-0.42) {$0$};
				\node at (0.7, -0.9) {$p_{12}$};
				\node at (0, 0.15) {$p_{13}$};
				\node at (-0.7, -0.9) {$p_{23}$};
				\node at (1.15, 0.1) {$p_{01}$};
				\node at (0, -1.87) {$p_{02}$};
				\node at (-1.15, 0.1) {$p_{03}$};
			\end{tikzpicture}
		}
		\subfigure[]{
			\begin{tikzpicture}[scale=1.8]
				\draw [-] (-1,0) -- (1,0);
				\draw [-] (-1,0) -- (0,-1.732);
				\draw [-] (1,0) -- (0,-1.732);
				\draw [-] (-1,0) -- (0.5,-0.866);
				\draw [-] (1,0) -- (-0.5,-0.866);
				\draw [-] (0,0) -- (0,-1.732);
				\draw[black] (0,-0.577) circle [x radius=0.577, y radius=0.577];
				\draw[black, fill] (0,0) circle [x radius=0.03, y radius=0.03];
				\draw[black, fill] (1,0) circle [x radius=0.03, y radius=0.03];
				\draw[black, fill] (-1,0) circle [x radius=0.03, y radius=0.03];
				\draw[black, fill] (0,-1.732) circle [x radius=0.03, y radius=0.03];
				\draw[black, fill] (-0.5,-0.866) circle [x radius=0.03, y radius=0.03];
				\draw[black, fill] (0.5,-0.866) circle [x radius=0.03, y radius=0.03];
				\draw[black, fill] (0,-0.577) circle [x radius=0.03, y radius=0.03];
				
				\node at (0.08,-0.42) {$0$};
				\node at (0.7, -0.9) {$p_{12}$};
				\node at (0, 0.15) {$0$};
				\node at (-0.7, -0.9) {$p_{23}$};
				\node at (1.15, 0.1) {$p_{01}$};
				\node at (0, -1.87) {$p_{02}$};
				\node at (-1.15, 0.1) {$p_{03}$};
			\end{tikzpicture}
		}
		\subfigure[]{
			\begin{tikzpicture}[scale=1.8]
				\draw [-] (-1,0) -- (1,0);
				\draw [-] (-1,0) -- (0,-1.732);
				\draw [-] (1,0) -- (0,-1.732);
				\draw [-] (-1,0) -- (0.5,-0.866);
				\draw [-] (1,0) -- (-0.5,-0.866);
				\draw [-] (0,0) -- (0,-1.732);
				\draw[black] (0,-0.577) circle [x radius=0.577, y radius=0.577];
				\draw[black, fill] (0,0) circle [x radius=0.03, y radius=0.03];
				\draw[black, fill] (1,0) circle [x radius=0.03, y radius=0.03];
				\draw[black, fill] (-1,0) circle [x radius=0.03, y radius=0.03];
				\draw[black, fill] (0,-1.732) circle [x radius=0.03, y radius=0.03];
				\draw[black, fill] (-0.5,-0.866) circle [x radius=0.03, y radius=0.03];
				\draw[black, fill] (0.5,-0.866) circle [x radius=0.03, y radius=0.03];
				\draw[black, fill] (0,-0.577) circle [x radius=0.03, y radius=0.03];
				
				\node at (0.08,-0.42) {$0$};
				\node at (0.7, -0.9) {$p_{12}$};
				\node at (0, 0.15) {$0$};
				\node at (-0.7, -0.9) {$p_{23}$};
				\node at (1.15, 0.1) {$p_{01}$};
				\node at (0, -1.87) {$0$};
				\node at (-1.15, 0.1) {$p_{03}$};
			\end{tikzpicture}
		}
		
		\subfigure[]{
			\begin{tikzpicture}[scale=1.8]
				\draw [-] (-1,0) -- (1,0);
				\draw [-] (-1,0) -- (0,-1.732);
				\draw [-] (1,0) -- (0,-1.732);
				\draw [-] (-1,0) -- (0.5,-0.866);
				\draw [-] (1,0) -- (-0.5,-0.866);
				\draw [-] (0,0) -- (0,-1.732);
				\draw[black] (0,-0.577) circle [x radius=0.577, y radius=0.577];
				\draw[black, fill] (0,0) circle [x radius=0.03, y radius=0.03];
				\draw[black, fill] (1,0) circle [x radius=0.03, y radius=0.03];
				\draw[black, fill] (-1,0) circle [x radius=0.03, y radius=0.03];
				\draw[black, fill] (0,-1.732) circle [x radius=0.03, y radius=0.03];
				\draw[black, fill] (-0.5,-0.866) circle [x radius=0.03, y radius=0.03];
				\draw[black, fill] (0.5,-0.866) circle [x radius=0.03, y radius=0.03];
				\draw[black, fill] (0,-0.577) circle [x radius=0.03, y radius=0.03];
				
				\node at (0.08,-0.42) {$0$};
				\node at (0.7, -0.9) {$0$};
				\node at (0, 0.15) {$p_{13}$};
				\node at (-0.7, -0.9) {$0$};
				\node at (1.15, 0.1) {$p_{01}$};
				\node at (0, -1.87) {$p_{02}$};
				\node at (-1.15, 0.1) {$p_{03}$};
			\end{tikzpicture}
		}
		\subfigure[]{
			\begin{tikzpicture}[scale=1.8]
				\draw [-] (-1,0) -- (1,0);
				\draw [-] (-1,0) -- (0,-1.732);
				\draw [-] (1,0) -- (0,-1.732);
				\draw [-] (-1,0) -- (0.5,-0.866);
				\draw [-] (1,0) -- (-0.5,-0.866);
				\draw [-] (0,0) -- (0,-1.732);
				\draw[black] (0,-0.577) circle [x radius=0.577, y radius=0.577];
				\draw[black, fill] (0,0) circle [x radius=0.03, y radius=0.03];
				\draw[black, fill] (1,0) circle [x radius=0.03, y radius=0.03];
				\draw[black, fill] (-1,0) circle [x radius=0.03, y radius=0.03];
				\draw[black, fill] (0,-1.732) circle [x radius=0.03, y radius=0.03];
				\draw[black, fill] (-0.5,-0.866) circle [x radius=0.03, y radius=0.03];
				\draw[black, fill] (0.5,-0.866) circle [x radius=0.03, y radius=0.03];
				\draw[black, fill] (0,-0.577) circle [x radius=0.03, y radius=0.03];
				
				\node at (0.08,-0.42) {$0$};
				\node at (0.7, -0.9) {$0$};
				\node at (0, 0.15) {$0$};
				\node at (-0.7, -0.9) {$0$};
				\node at (1.15, 0.1) {$p_{01}$};
				\node at (0, -1.87) {$p_{02}$};
				\node at (-1.15, 0.1) {$p_{03}$};
			\end{tikzpicture}
		}
	\caption{\label{fig:conorm_diagrams}Conorm diagrams of the 5 families of Voronoi cells.}
	\end{center}
\end{figure}
\renewcommand{\thesubfigure}{(\alph{subfigure})}

\begin{definition}
	We use the term \emph{combinatorial type} of a rank 3 lattice to refer to which of the above 5 parallelohedra represents the Voronoi cell of the lattice.
\end{definition}

We now work out conorm diagrams for the families of lattices we will encounter in studying the shapes of Galois quartic fields.

\subsection{Tetragonal and cubic lattices}
In this section, we determine the conorm diagrams of the tetragonal and cubic lattices, the latter being a special case of the former. We being by recalling what tetragonal and cubic lattices are.

Consider a right rectangular prism of height $c$ with square base of side $a$, with $a\neq c$. A \emph{primitive tetragonal lattice} ($tP$) consists of the vertices of this prism together with all its translates that tile space. A \emph{body-centered tetragonal lattice} ($tI$) is like a primitive one, but with the centre of each prism added to the set of lattice points and $c\neq\sqrt{2}a$. A \textit{primitive} (resp.\ \textit{body-centered}) \textit{cubic lattice} ($cP$ and $cI$, respectively) is as above, but with $a=c$. A \textit{face-centered cubic lattice} ($cF$) is obtained from a primitive one by adding the centre of each face of the prism to the set of lattice points; it is, in fact, the same as the body-centered tetragonal lattice with $c=\sqrt{2}a$.

\begin{proposition}
	There are two combinatorial types of body-centered tetragonal lattices depending on whether ${\frac{c}{a}<\sqrt{2}}$ or ${\frac{c}{a}>\sqrt{2}}$. Their conorm diagrams are given in Figure~\ref{fig:conorms_tI}. The body-centered cubic lattice is obtained by setting $\frac{c}{a}=1$ and the face-centered cubic by setting $\frac{c}{a}=\sqrt{2}$.
	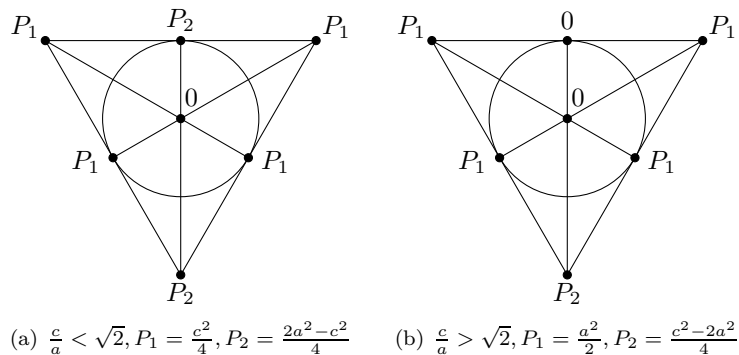
\begin{figure}[h]
	\begin{center}
		\subfigure[$\frac{c}{a}<\sqrt{2}, P_1=\frac{c^2}{4}, P_2=\frac{2a^2-c^2}{4}$]{
			\begin{tikzpicture}[scale=1.8]
				\draw [-] (-1,0) -- (1,0);
				\draw [-] (-1,0) -- (0,-1.732);
				\draw [-] (1,0) -- (0,-1.732);
				\draw [-] (-1,0) -- (0.5,-0.866);
				\draw [-] (1,0) -- (-0.5,-0.866);
				\draw [-] (0,0) -- (0,-1.732);
				\draw[black] (0,-0.577) circle [x radius=0.577, y radius=0.577];
				\draw[black, fill] (0,0) circle [x radius=0.03, y radius=0.03];
				\draw[black, fill] (1,0) circle [x radius=0.03, y radius=0.03];
				\draw[black, fill] (-1,0) circle [x radius=0.03, y radius=0.03];
				\draw[black, fill] (0,-1.732) circle [x radius=0.03, y radius=0.03];
				\draw[black, fill] (-0.5,-0.866) circle [x radius=0.03, y radius=0.03];
				\draw[black, fill] (0.5,-0.866) circle [x radius=0.03, y radius=0.03];
				\draw[black, fill] (0,-0.577) circle [x radius=0.03, y radius=0.03];
				
				\node at (0.08,-0.42) {$0$};
				\node at (0.7, -0.9) {$P_1$};
				\node at (0, 0.15) {$P_2$};
				\node at (-0.7, -0.9) {$P_1$};
				\node at (1.15, 0.1) {$P_1$};
				\node at (0, -1.87) {$P_2$};
				\node at (-1.15, 0.1) {$P_1$};
			\end{tikzpicture}
		}
		\subfigure[$\frac{c}{a}>\sqrt{2}, P_1=\frac{a^2}{2}, P_2=\frac{c^2-2a^2}{4}$]{
			\begin{tikzpicture}[scale=1.8]
				\draw [-] (-1,0) -- (1,0);
				\draw [-] (-1,0) -- (0,-1.732);
				\draw [-] (1,0) -- (0,-1.732);
				\draw [-] (-1,0) -- (0.5,-0.866);
				\draw [-] (1,0) -- (-0.5,-0.866);
				\draw [-] (0,0) -- (0,-1.732);
				\draw[black] (0,-0.577) circle [x radius=0.577, y radius=0.577];
				\draw[black, fill] (0,0) circle [x radius=0.03, y radius=0.03];
				\draw[black, fill] (1,0) circle [x radius=0.03, y radius=0.03];
				\draw[black, fill] (-1,0) circle [x radius=0.03, y radius=0.03];
				\draw[black, fill] (0,-1.732) circle [x radius=0.03, y radius=0.03];
				\draw[black, fill] (-0.5,-0.866) circle [x radius=0.03, y radius=0.03];
				\draw[black, fill] (0.5,-0.866) circle [x radius=0.03, y radius=0.03];
				\draw[black, fill] (0,-0.577) circle [x radius=0.03, y radius=0.03];
				
				\node at (0.08,-0.42) {$0$};
				\node at (0.7, -0.9) {$P_1$};
				\node at (0, 0.15) {$0$};
				\node at (-0.7, -0.9) {$P_1$};
				\node at (1.15, 0.1) {$P_1$};
				\node at (0, -1.87) {$P_2$};
				\node at (-1.15, 0.1) {$P_1$};
			\end{tikzpicture}
		}
	\caption{\label{fig:conorms_tI}Conorm diagrams of body-centered tetragonal lattices. The body-centered cubic lattice is obtained by taking $P_1=P_2$ in the diagram on the left, while the face-centered cubic lattice is obtained by taking $P_2=0$ in either diagram.}
	\end{center}
\end{figure}
\end{proposition}
\begin{proof}
	Independent of the value of $c/a$, the vectors $w_1=(c,0,0),w_2=(0,a,0),w_3=(0,0,a)$ form a basis of a primitive tetragonal lattice. The associated body-centered lattice is then the $\ZZ$-span of $w_1,w_2,w_3$, and the midpoint $w_4=\frac{1}{2}(c,a,a)$. Let
	\begin{align*}
		v_0&=\frac{1}{2}(c,a,a)&v_1&=\frac{1}{2}(-c,-a,a)&v_2&=\frac{1}{2}(c,-a,-a)&v_3&=\frac{1}{2}(-c,a,-a).
	\end{align*}
	Since
	\begin{align*}
		w_1&=-v_1-v_3&		&&v_1&=w_3-w_4\\
		w_2&=-v_1-v_2&\text{and}&&v_2&=w_1-w_4\\
		w_3&=-v_2-v_3&		&&v_3&=w_2-w_4\\
		w_4&=v_0=-v_1-v_2-v_3
	\end{align*}
	we see that $(v_0,v_1,v_2,v_3)$ is an obtuse superbase of the body-centered tetragonal lattice. Computing the putative conorm diagram shows that it is already obtuse when $c\leq\sqrt{2}a$. When $c>\sqrt{2}a$, applying one step of the Voronoi reduction algorithm to the vertical line and an appropriate automorphism of the Fano plane yields the desired result.
\end{proof}

\begin{proposition}
	The unique family of primitive tetragonal lattices has conorm diagram given in Figure~\ref{fig:conorms_tP}. The primitive cubic lattice is obtained by taking $a=c$.
\end{proposition}
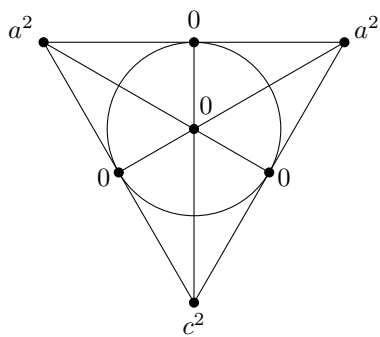
\begin{figure}[h]
	\begin{center}
		\subfigure{
			\begin{tikzpicture}[scale=2]
				\draw [-] (-1,0) -- (1,0);
				\draw [-] (-1,0) -- (0,-1.732);
				\draw [-] (1,0) -- (0,-1.732);
				\draw [-] (-1,0) -- (0.5,-0.866);
				\draw [-] (1,0) -- (-0.5,-0.866);
				\draw [-] (0,0) -- (0,-1.732);
				\draw[black] (0,-0.577) circle [x radius=0.577, y radius=0.577];
				\draw[black, fill] (0,0) circle [x radius=0.03, y radius=0.03];
				\draw[black, fill] (1,0) circle [x radius=0.03, y radius=0.03];
				\draw[black, fill] (-1,0) circle [x radius=0.03, y radius=0.03];
				\draw[black, fill] (0,-1.732) circle [x radius=0.03, y radius=0.03];
				\draw[black, fill] (-0.5,-0.866) circle [x radius=0.03, y radius=0.03];
				\draw[black, fill] (0.5,-0.866) circle [x radius=0.03, y radius=0.03];
				\draw[black, fill] (0,-0.577) circle [x radius=0.03, y radius=0.03];
				
				\node at (0.08,-0.42) {0};
				\node at (0.6, -0.9) {0};
				\node at (0, 0.15) {0};
				\node at (-0.6, -0.9) {0};
				\node at (1.15, 0.1) {$a^2$};
				\node at (0, -1.87) {$c^2$};
				\node at (-1.15, 0.1) {$a^2$};
			\end{tikzpicture}
		}
	\caption{\label{fig:conorms_tP}Conorm diagram of the primitive tetragonal lattice. The primitive cubic lattice is obtained by taking $a=c$.}
	\end{center}
\end{figure}
\begin{proof}
	The obtuse superbase given by
	\begin{align*}
		v_0&=(-c,-a,-a)&v_1&=(0,a,0)&v_2&=(c,0,0)&v_3&=(0,0,a)
	\end{align*}
	yields the claimed conorm diagram.
\end{proof}

\subsection{Orthorhombic and hexagonal lattices}
Consider a right rectangular prism of height $c$ whose base is a non-square rectangle. Let $a$ denote its depth and $b$ its width, labelled so that $a\leq b$.\footnote{Technically, when $a=b$ what we have is a tetragonal lattice. Base-centered tetragonal lattices will occur later on in a family with base-centered orthorhombic lattices, so we allow $a=b$ here. Note that a base-centered tetragonal lattice with base of side $a$ is the same as a primitive tetragonal lattice with base of side $a/\sqrt{2}$.} Assume $a\neq c\neq b$. A \emph{primitive orthorhombic lattice} ($oP$) consists of the vertices of this prism together with all its translates that tile space. If we add a lattice point in the centre of every base, we get a \emph{base-centered orthorhombic lattice} ($oC$). A special case occurs when $b=\sqrt{3}a$: a \emph{primitive hexagonal lattice} ($hP$). Equivalently, take a right prism of height $c$ whose base is a regular hexagon of side $a$, together with a lattice point in the centre of each base. These latter lattice points are the vertices of the rectangular prisms of the ($oC$, $b=\sqrt{3}a$). The body-centered orthorhombic lattice ($oI$) is obtained similarly to the body-centered tetragonal lattice, i.e.\ by taking the primitive orthorhombic lattice and adding the centre of each prism to the lattice. We choose the side lengths so that $a<b<c$ in this case.

\begin{proposition}\label{prop:baseconorm}
	The unique family of base-centered orthorhombic lattices has conorm diagram given in Figure~\ref{fig:conorms_oC}. The primitive hexagonal lattice is obtained by taking $b=\sqrt{3}a$.
\end{proposition}
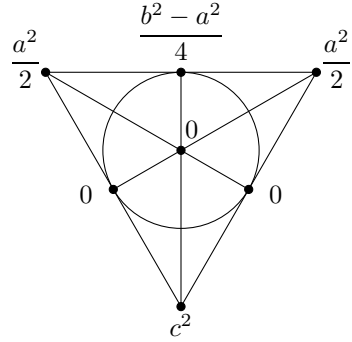
\begin{figure}[h]
	\begin{center}
		\subfigure{
			\begin{tikzpicture}[scale=1.8]
				\draw [-] (-1,0) -- (1,0);
				\draw [-] (-1,0) -- (0,-1.732);
				\draw [-] (1,0) -- (0,-1.732);
				\draw [-] (-1,0) -- (0.5,-0.866);
				\draw [-] (1,0) -- (-0.5,-0.866);
				\draw [-] (0,0) -- (0,-1.732);
				\draw[black] (0,-0.577) circle [x radius=0.577, y radius=0.577];
				\draw[black, fill] (0,0) circle [x radius=0.03, y radius=0.03];
				\draw[black, fill] (1,0) circle [x radius=0.03, y radius=0.03];
				\draw[black, fill] (-1,0) circle [x radius=0.03, y radius=0.03];
				\draw[black, fill] (0,-1.732) circle [x radius=0.03, y radius=0.03];
				\draw[black, fill] (-0.5,-0.866) circle [x radius=0.03, y radius=0.03];
				\draw[black, fill] (0.5,-0.866) circle [x radius=0.03, y radius=0.03];
				\draw[black, fill] (0,-0.577) circle [x radius=0.03, y radius=0.03];
				
				\node at (0.08,-0.42) {$0$};
				\node at (0.7, -0.9) {$0$};
				\node at (0, 0.3) {$\dfrac{b^2-a^2}{4}$};
				\node at (-0.7, -0.9) {$0$};
				\node at (1.15, 0.1) {$\dfrac{a^2}{2}$};
				\node at (0, -1.87) {$c^2$};
				\node at (-1.15, 0.1) {$\dfrac{a^2}{2}$};
			\end{tikzpicture}
		}
	\caption{\label{fig:conorms_oC}Conorm diagram of the base-centered orthorhombic lattice. The primitive hexagonal lattice is obtained by taking $b=\sqrt{3}a$. When $a=b$, we obtain a primitive tetragonal lattice with base of side $a/\sqrt{2}$.}
	\end{center}
\end{figure}
\begin{proof}
	An obtuse superbase is given by
	\begin{align*}
		v_0&=(-a,0,-c)&v_1&=\frac{1}{2}(a,-b,0)&v_2&=(0,0,c)&v_3&=\frac{1}{2}(a,b,0).
	\end{align*}
\end{proof}

\begin{proposition}\label{prop:bodyconorm}
There are three combinatorial types of body-centered orthorhombic lattices depending on whether $a^2+b^2$ is less than, equal to, or greater than $c^2$. Their conorm diagrams are given in Figure~\ref{fig:conorms_oI}.
	\begin{figure}[h]
	\begin{center}
		\subfigure[$a^2+b^2\leq c^2$]{
			\begin{tikzpicture}[scale=1.8]
				\draw [-] (-1,0) -- (1,0);
				\draw [-] (-1,0) -- (0,-1.732);
				\draw [-] (1,0) -- (0,-1.732);
				\draw [-] (-1,0) -- (0.5,-0.866);
				\draw [-] (1,0) -- (-0.5,-0.866);
				\draw [-] (0,0) -- (0,-1.732);
				\draw[black] (0,-0.577) circle [x radius=0.577, y radius=0.577];
				\draw[black, fill] (0,0) circle [x radius=0.03, y radius=0.03];
				\draw[black, fill] (1,0) circle [x radius=0.03, y radius=0.03];
				\draw[black, fill] (-1,0) circle [x radius=0.03, y radius=0.03];
				\draw[black, fill] (0,-1.732) circle [x radius=0.03, y radius=0.03];
				\draw[black, fill] (-0.5,-0.866) circle [x radius=0.03, y radius=0.03];
				\draw[black, fill] (0.5,-0.866) circle [x radius=0.03, y radius=0.03];
				\draw[black, fill] (0,-0.577) circle [x radius=0.03, y radius=0.03];
				
				\node at (0.08,-0.42) {$0$};
				\node at (0.7, -0.9) {$\dfrac{a^2}{2}$};
				\node at (0, 0.15) {$0$};
				\node at (-0.7, -0.9) {$\dfrac{b^2}{2}$};
				\node at (1.15, 0.1) {$\dfrac{a^2}{2}$};
				\node at (0, -2) {$\dfrac{c^2-(a^2+b^2)}{4}$};
				\node at (-1.15, 0.1) {$\dfrac{b^2}{2}$};
			\end{tikzpicture}
		}
		\subfigure[$a^2+b^2\geq c^2$]{
			\begin{tikzpicture}[scale=1.8]
				\draw [-] (-1,0) -- (1,0);
				\draw [-] (-1,0) -- (0,-1.732);
				\draw [-] (1,0) -- (0,-1.732);
				\draw [-] (-1,0) -- (0.5,-0.866);
				\draw [-] (1,0) -- (-0.5,-0.866);
				\draw [-] (0,0) -- (0,-1.732);
				\draw[black] (0,-0.577) circle [x radius=0.577, y radius=0.577];
				\draw[black, fill] (0,0) circle [x radius=0.03, y radius=0.03];
				\draw[black, fill] (1,0) circle [x radius=0.03, y radius=0.03];
				\draw[black, fill] (-1,0) circle [x radius=0.03, y radius=0.03];
				\draw[black, fill] (0,-1.732) circle [x radius=0.03, y radius=0.03];
				\draw[black, fill] (-0.5,-0.866) circle [x radius=0.03, y radius=0.03];
				\draw[black, fill] (0.5,-0.866) circle [x radius=0.03, y radius=0.03];
				\draw[black, fill] (0,-0.577) circle [x radius=0.03, y radius=0.03];
				
				\node at (0.08,-0.42) {$0$};
				\node at (0.7, -0.9) {$P_2$};
				\node at (0, 0.15) {$P_1$};
				\node at (-0.7, -0.9) {$P_3$};
				\node at (1.15, 0.1) {$P_3$};
				\node at (0, -1.87) {$P_1$};
				\node at (-1.15, 0.1) {$P_2$};
				\node at (0, -2) {$\phantom{\dfrac{c^2-(a^2+b^2)}{4}}$};
			\end{tikzpicture}
		}
	\captionsetup{singlelinecheck=off}
	\caption[.]{\label{fig:conorms_oI}Conorm diagrams of body-centered orthorhombic lattices. Here,
	\begin{equation*}
		P_1=\dfrac{-a^2+b^2+c^2}{4}, P_2=\dfrac{a^2-b^2+c^2}{4}, P_3=\dfrac{a^2+b^2-c^2}{4}.
	\end{equation*}
	}
	\end{center}
\end{figure}
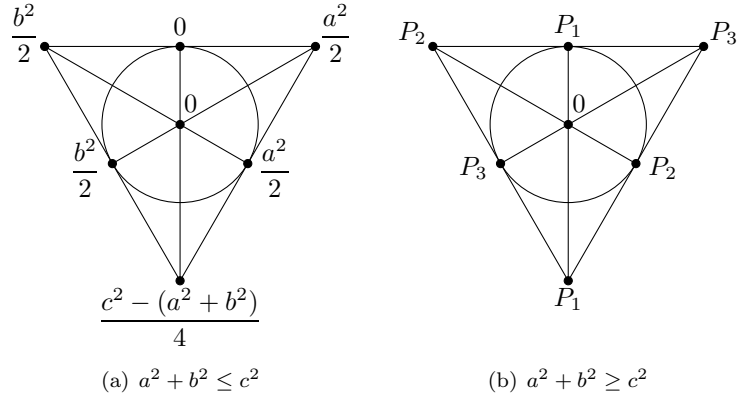
\end{proposition}
\begin{proof}
	When $a^2+b^2\leq c^2$, an obtuse superbase is given by
	\begin{align*}
		v_0&=\frac{1}{2}(a,b,c)&v_1&=(-a,0,0)&v_2&=\frac{1}{2}(a,b,-c)&v_3&=(0,-b,0).
	\end{align*}
	When $a^2+b^2\geq c^2$, an obtuse superbase is given by
	\begin{align*}
		v_0&=\frac{1}{2}(a,b,c)&v_1&=\frac{1}{2}(-a,-b,c)&v_2&=\frac{1}{2}(a,-b,-c)&v_3&=\frac{1}{2}(-a,b,-c).
	\end{align*}
\end{proof}

\subsection{Measures on spaces of orthorhombic lattices}
The orthorhombic lattices discussed in the previous section will arise as shapes of $V_4$-quartic fields and we will show that the shapes of $V_4$ -quartic fields are equidistributed within certain spaces of orthorhombic lattices. Such a statement requires that we define measures on these spaces. These measures will be inherited from a natural group action by diagonal matrices. This material is only needed in \S\ref{sec:V4equid} for the proof of Theorem~\ref{thm:V4equid}.

We begin by parametrizing the space of shapes of base-centered orthorhombic lattices, i.e.\ we study these lattices up to rotations, reflections, and scaling. Let $\Lambda$ be a base-centered orthorhombic lattice. After some possible rotations (that align the sides of the rectangular prism inside this lattice with the coordinate axes), we may assume that
\begin{align*}
		v_1&=\frac{1}{2}(a,-b,0)&v_2&=(0,0,c)&v_3&=\frac{1}{2}(a,b,0),
\end{align*} is a basis of $\Lambda$ as in Proposition~\ref{prop:baseconorm}. If $a>b$, we may apply a reflection so that, without loss of generality, we may assume that $a\leq b$. Since the edges of length $a$ and $b$ are distinguished from $c$ (since they form the base of the rectangular prism), it is natural to scale $\Lambda$ so that its height is $1$. In other words, we will take as parameters $a/c$ and $b/c$. Accordingly, let
\[
	G_{oC}(x,y):=\vect{\dfrac{x^2+y^2}{4}&0&\dfrac{x^2-y^2}{4}\\[5pt]0&1&0\\\dfrac{x^2-y^2}{4}&0&\dfrac{x^2+y^2}{4}},
\]
\[
	\mc{G}_{oC}:=\left\{G_{oC}(x,y):0<x\leq y\right\}\subseteq\mc{G},
\]
and
\[
	\mc{S}_{oC}:=\left\{\gl_3(\ZZ)\cdot G\cdot\RR^\times:G\in\mc{G}_{oC}\right\}.
\]
We now show that these Gram matrices give a complete set of representatives of the shapes of base-centered orthorhombic lattices.
\begin{proposition}
	The map $\mc{G}_{oC}\rightarrow\mc{S}_{oC}\subseteq\gl_3(\ZZ)\backslash\mc{G}/\RR^\times$ sending $G_{oC}(x,y)$ to
	\[
		\sh_{oC}(x,y):=\gl_3(\ZZ)\cdot G_{oC}(x,y)\cdot\RR^\times
	\]
	is a bijection and $\mc{S}_{oC}$ is the space of shapes of base-centered orthorhombic lattices. If $\Lambda$ has base with sides of length $a$ and $b$, and height $c$, then the shape of $\Lambda$ is $\sh_{oC}(a/c,b/c)$ or $\sh_{oC}(b/c,a/c)$ according to whether $a\leq b$ or $a\geq b$.
\end{proposition}
\begin{proof}
	The discussion above this proposition describes how, starting from an arbitrary base-centered orthorhombic lattice $\Lambda$ whose base has sides of length $a$ and $b$ and whose height is $c$, we may apply rotations, reflections, and scalings to get an equivalent lattice whose Gram matrix is $G_{oC}(a/c,b/c)$ or $G_{oC}(b/c,a/c)$ according to whether $a\leq b$ or $a\geq b$. Conversely, one can write down the lattice whose shape is a given element of $\mc{S}_{oC}$. This shows that $\mc{S}_{oC}$ is indeed the space of shapes of base-centered orthorhombic lattices.

	Now suppose that $\gl_3(\ZZ)\cdot G_{oC}(x_1,y_1)\cdot\RR^\times=\gl_3(\ZZ)\cdot G_{oC}(x_2,y_2)\cdot\RR^\times$ (where $x_1\leq y_1$ and $x_2\leq y_2$). These two Gram matrices correspond to lattices whose conorm diagrams $\mc{C}_1$ and $\mc{C}_2$ are given as in Figure~\ref{fig:conorms_oC} with $a=x_i$, $b=y_i$, and $c=1$. There is a unique line in $\mc{C}_i$ all of whose points have a non-zero label, so that any automorphism of the Fano plane bringing $\mc{C}_1$ to $\mc{C}_2$ must fix this line. There is a unique point in $\mc{C}_1$ not on this line with a non-zero label, so that this point must also be fixed. Since the label on this point in both $\mc{C}_1$ and $\mc{C}_2$ is $1$, no scaling can occur. At least two of the non-zero labels on the fixed line in $\mc{C}_1$ are equal and so must match the (at least) two equal labels on the fixed line in $\mc{C}_2$. This implies that $x_1^2=x_2^2$. Since $x_i>0$, this forces them to be equal. The remaining non-zero entry then forces $y_1=y_2$.
\end{proof}

\begin{definition}
	We define the measure $\mu_{oC}$ on $\mc{S}_{oC}$ by
	\[
		d\mu_{oC}(x,y):=d^\times xd^\times y=\dfrac{dxdy}{xy},
	\]
	where $x$ and $y$ are the coordinates in $\sh_{oC}(x,y)$ and $dxdy$ denotes the usual Lebesgue measure on a subset of $\RR^2$.
\end{definition}
\begin{remark}\label{rem:invariantmeasure}
	The motivation for this definition is that this measure is the one inherited from the natural group action on the space of shapes. For instance, the whole space $\gl_d(\ZZ)\backslash\mc{G}/\RR^\times$ inherits a natural measure from the action of $\gl_d(\RR)$ on $\mc{G}$. For $\mc{S}_{oC}$, note that $\mc{G}_{oC}$ is a ``translate'' of the Gram matrices of primitive orthorhombic lattices. Indeed, let
	\[
		G(x,y):=\vect{x^2\\&y^2\\&&1}
	\]
	and let
	\[
		P_{oC}:=\vect{1/2&1/2&0\\0&0&1\\1/2&-1/2&0}.
	\]
	Then,
	\[
		P_{oC}\cdot G(x,y)=G_{oC}(x,y).
	\]
	The set of Gram matrices of the form $G(x,y)$ is an orbit of the group
	\[
		\mc{T}:=\left\{G(a,b):a,b\in\RR_{>0}\right\}\leq\gl_3(\RR);
	\]
	indeed $G(\sqrt{x},\sqrt{y})\cdot G(1,1)=G(x,y)$. Similarly, the set of Gram matrices of the form $G_{oC}(x,y)$ is an orbit of the isomorphic group $\mc{T}_{oC}:=P\mc{T}P^{-1}$; indeed, for $T=G(\sqrt{x},\sqrt{y})$,
	\begin{align*}
		PTP^{-1}\cdot G_{oC}(1,1)	&=PTP^{-1}(PG(1,1)P^{T})P^{-T}T^TP^T\\
								&=PTG(1,1)T^TP^T\\
								&=P\cdot (T\cdot G(1,1))\\
								&=G_{oC}(x,y).
	\end{align*}
	The Haar measure on the group $\mc{T}$ (and hence also $\mc{T}_{oC}$) is (any positive multiple of) $d^\times x d^\times y$, where the parameters $x,y$ give the element
	\[
		\vect{x\\&y\\&&1}\in\mc{T}.
	\]
	It is then natural to transfer this measure over to the subset $\mc{G}_{oC}$ of the (free) orbit of $\mc{T}_{oC}$, as we have done in the definition above.
\end{remark}

Recall that equidistribution is a statement about weak convergence of a sequence of measures and recall that a sequence of measures $\{\mu_n\}$ on $\mc{S}_{oC}$ converges weakly to $\mu_{oC}$ if for all $f\in C_c(\mc{S}_{oC})$ (the continuous functions with compact support),
\[
	\lim_{n\rightarrow\infty}\int_{\mc{S}_oC}fd\mu_n=\int_{\mc{S}_oC}fd\mu_{oC}.
\]
Since we will be counting number fields in an explicit way below, we must simplify our lives when it comes to which kinds of functions we need to test this convergence on. For two positive real numbers $R_1<R_2$, let
\[
	W_{oC}(R_1,R_2)=\left\{\sh_{oC}(x,y):R_1\leq x\leq y<R_2\right\}
\]
and let $\chi_{oC,R_1,R_2}$ denote its characteristic function. We now show that it is sufficient to test these functions for the purposes of proving equidistribution.
\begin{lemma}\label{lem:WoCsuffice}
	Suppose that
	\[
		\lim_{n\rightarrow\infty}\int_{\mc{S}_{oC}}\chi_{oC,R_1,R_2}d\mu_n=\int_{\mc{S}_{oC}}\chi_{oC,R_1,R_2}d\mu_{oC}.
	\]
	for all $R_1,R_2\in\RR_{>0}$ with $R_1<R_2$. Then, $\mu_n$ converges weakly to $\mu_{oC}$.
\end{lemma}
\begin{proof}
	Recall that for the usual Lebesgue measure on $\RR_{>0}^2$, any continuous function $f$ with compact support can be`` approximated'' above and below by two ``step functions'', i.e.\ for every $f$ and for every $\epsilon>0$, there are two functions $f_1$ and $f_2$ that are finite linear combinations of characteristic functions of squares such that $f_1\leq f\leq f_2$ and
	\[
		\int_{\RR_{>0}^2}(f_2-f_1)dxdy<\epsilon.
	\]
	Since the measure $\mu_{oC}$ is absolutely continuous with respect the Lebesgue measure on $\RR_{>0}^2$, this is still true for it, where the ``squares'' are replaced by their intersection with the set $\{x\leq y\}$. A straightforward proof as in \cite[Theorem~3.1]{PureCubicShapes} then shows that it suffices to test convergence on these ``squares''. To prove this lemma, it now suffices to show that the characteristic functions of these ``squares'' are finite linear combinations of the $\chi_{oC,R_1,R_2}$.
	
	So, let $\mc{C}$ be a ``square'' in $\mc{S}_{oC}$ whose vertices are $(x_0,y_0),(x_0+r,y_0),(x_0+r,y_0-r),$ and ${(x_0,y_0-r)}$, and let $\chi_\mc{C}$ denote its characteristic function. A simple inclusion-exclusion shows that
	\[
		\chi_\mc{C}=\chi_{oC,x_0,y_0}-\chi_{oC,x_0+r,y_0}-\chi_{oC,x_0,y_0-r}+\chi_{oC,x_0+r,y_0-r}.
	\]
\end{proof}

The following result will therefore be useful in \S\ref{sec:V4equid} below.
\begin{lemma}\label{lem:WoCmeasure}
	For $R_1<R_2\in\RR_{>0}$,
	\[
		\mu_{oC}(W_{oC}(R_1,R_2))=\frac{1}{2}(\log(R_2)-\log(R_1))^2.
	\]
\end{lemma}
\begin{proof}
	We have that
	\begin{align}
		\int_{W_{oC}(R_1,R_2)}d\mu_{oC}	&=\int_{R_1}^{R_2}\int_{R_1}^y\dfrac{1}{xy}dxdy\label{eqn:WoCintegralFirst}\\
										&=\int_{R_1}^{R_2}\dfrac{\log(y)-\log(R_1)}{y}dy\nonumber\\
										&=\int_{\log(R_1)}^{\log(R_2)}udu-\log(R_1)(\log(R_2)-\log(R_1))\nonumber\\
										&=\dfrac{(\log R_2)^2-(\log R_1)^2}{2}-\log(R_1)\log(R_2)+(\log R_1)^2\nonumber\\
										&=\frac{1}{2}(\log(R_2)-\log(R_1))^2\nonumber.
	\end{align}
\end{proof}

We now proceed analogously for body-centered lattices. Let $\Lambda$ be a body-centered orthorhombic lattice. After some possible rotations as above, we may assume that
\begin{align*}
		v_1&=\frac{1}{2}(-a,-b,c)&v_2&=\frac{1}{2}(a,-b,-c)&v_3&=\frac{1}{2}(-a,b,-c)
\end{align*} is a basis of $\Lambda$ as in Proposition~\ref{prop:bodyconorm}.\footnote{Even if $a^2+b^2\leq c^2$, this is still a basis, though not part of an obtuse subperbase.} By applying reflections, we may assume, without loss of generality, that $a<b<c$. We take as parameters $a/c$ and $b/c$ like above. Accordingly, let
\[
	G_{oI}(x,y):=\vect{\dfrac{x^2+y^2+1}{4}&\dfrac{-x^2+y^2-1}{4}&\dfrac{x^2-y^2-1}{4}\\[5pt]\dfrac{-x^2+y^2-1}{4}&\dfrac{x^2+y^2+1}{4}&\dfrac{-x^2-y^2+1}{4}\\[5pt]\dfrac{x^2-y^2-1}{4}&\dfrac{-x^2-y^2+1}{4}&\dfrac{x^2+y^2+1}{4}},
\]
\[
	\mc{G}_{oI}:=\left\{G_{oI}(x,y):0<x<y<1\right\}\subseteq\mc{G},
\]
and
\[
	\mc{S}_{oI}:=\left\{\gl_3(\ZZ)\cdot G\cdot\RR^\times:G\in\mc{G}_{oI}\right\}.
\]
Similarly to above, we show that $\mc{G}_{oI}$ is  a complete set of representatives of the shapes of body-centered orthorhombic lattices.
\begin{proposition}
	The map $\mc{G}_{oI}\rightarrow\mc{S}_{oI}\subseteq\gl_3(\ZZ)\backslash\mc{G}/\RR^\times$ sending $G_{oI}(x,y)$ to
	\[
		\sh_{oI}(x,y):=\gl_3(\ZZ)\cdot G_{oI}(x,y)\cdot\RR^\times
	\]
	is a bijection and $\mc{S}_{oI}$ is the space of shapes of body-centered orthorhombic lattices. If $\Lambda$ has base with sides of length $a,b,$ and $c$, then its shape is $\sh_{oI}(x,y)$ for exactly one pair
	\[
		(x,y)\in\left\{(a/c,b/c),(b/c,a/c),(a/b,c/b),(c/b,a/b),(b/a,c/a),(c/a,b/a)\right\},
	\]
	whichever gives $x<y<1$.
\end{proposition}
\begin{proof}
	It is explained above, how to apply rotations, reflections, and scalings to a body-centered orthorhombic lattice to get an equivalent lattice whose Gram matrix is $G_{oI}(x,y)$ with $x<y<1$ and $(x,y)$ one of the pairs listed in the statement of this proposition. One can conversely construct a lattice for any element of $\mc{S}_{oI}$. This shows that $\mc{S}_{oI}$ is indeed the space of shapes of body-centered orthorhombic lattices.

	Now suppose that $\gl_3(\ZZ)\cdot G_{oI}(x_1,y_1)\cdot\RR^\times=\gl_3(\ZZ)\cdot G_{oI}(x_2,y_2)\cdot\RR^\times$ (where $x_1<y_1<1$ and $x_2<y_2<1$). There are two combinatorial types of conorm diagram in Figure~\ref{fig:conorms_oI}. Accordingly, the equality of these shapes implies that either $x_i^2+y_i^2<1$ for both values of $i$ or $x_i^2+y_i^2>1$ for both values of $i$ (for instance, one type of diagram has more zeroes than the other, so that no automorphism of the Fano plane can bring one to the other).
	
	First consider when $x_i^2+y_i^2>1$. There are exactly three lines with two non-zero labels in the corresponding conorm diagram. Picking two of these at a time and summing all the labels on these two yields $1,x^2,$ and $y^2$, respectively, so that $x$ and $y$ are determined by the conorm diagram.
	
	Now, consider when $x_i^2+y_i^2<1$. The Gram matrices we have picked are not those associated to the conorm diagrams, however the matrix
	\[
		P=\vect{1&0&1\\-1&0&0\\1&1&0}\in\gl_3(\ZZ)
	\]
	acts on $G_{oI}(x,y)$ bringing it to
	\[
		G_{oI}^\prime(x,y):=\vect{x^2&-\dfrac{x^2}{2}&0\\-\dfrac{x^2}{2}&\dfrac{x^2+y^2+1}{4}&-\dfrac{y^2}{2}\\0&-\dfrac{y^2}{2}&y^2}.
	\]
	Note that, since $P\in\gl_3(\ZZ)$, the $G_{oI}(x_i,y_i)$ give the same shape if and only if the $G_{oI}^\prime(x_i,y_i)$ do. This matrix $G_{oI}^\prime(x,y)$ has associated conorm diagram that on the left of Figure~\ref{fig:conorms_oI} (with $a=x,b=y,$ and $c=1$). Such a conorm diagram has exactly two lines all of whose labels are non-zero. On each of these lines, two of the labels are equal and given by $x^2/2$ and $y^2/2$, respectively. An automorphism of the Fano plane must bring one line to itself or to the other, but since $x_i<y_i$ it cannot switch the lines. Once again, the $x_i$ and the $y_i$ can be read off from the conorm diagrams, so that $x_1=x_2$ and $y_1=y_2$.
\end{proof}

Analogues of Remark~\ref{rem:invariantmeasure} and Lemma~\ref{lem:WoCsuffice} hold in the case of body-centered orthorhombic lattices. Accordingly, we are led to the following definitions.
\begin{definition}\mbox{}
	\begin{enumerate}
		\item We define the measure $\mu_{oI}$ on $\mc{S}_{oI}$ by
		\[
			d\mu_{oI}(x,y):=d^\times xd^\times y=\dfrac{dxdy}{xy},
		\]
		where $dxdy$ denotes the usual Lebesgue measure on a subset of $\RR^2$.
		\item For $R_1<R_2\in(0,1)$, let
		\[
			W_{oI}(R_1,R_2):=\left\{\sh_{oI}(x,y):R_1\leq x< y< R_2\right\}.
		\]
	\end{enumerate}
\end{definition}
The same calculation as in Lemma~\ref{lem:WoCsuffice} yields that
\begin{equation}
	\mu_{oI}(W_{oI}(R_1,R_2))=\frac{1}{2}(\log(R_2)-\log(R_1))^2.
\end{equation}


\section{The shapes of $V_4$-quartic fields}\label{sec:V4_shapes}
In this section, we determine the shapes of Galois quartic extensions of $\QQ$ whose Galois group is the Klein $4$-group $V_4$. Such a field $K$ is determined by its 3 quadratic subfields $\QQ(\sqrt{D_i}), i=1,2,3$, where we take $D_i$ squarefree. Note that for $\{i,j,k\}=\{1,2,3\}$ and $g_k=\gcd(D_i,D_j)$,\footnote{If any of the $D$'s are negative then exactly two of them are; in this case, we would choose $g_k<0$ if $D_i$ and $D_j$ are the negative ones.} we have that
\begin{eqnarray}
	&\ds D_k=\frac{D_iD_j}{g_k^2},\\
	&D_k=g_ig_j,\\
	&D_1D_2D_3>0,\\
	&g_1,g_2,g_3\text{ are squarefree and pairwise relatively prime}.
\end{eqnarray}
The work of Kenneth S.\ Williams (\cite{WilliamsBiquad}) breaks the question of integral bases of these fields into 3 cases:
\begin{itemize}
	\item[(i)] $\{D_1,D_2,D_3\}\equiv\{2,2,3\}\mod{4}$;
	\item[(ii)] $\{D_1,D_2,D_3\}\equiv\{1,2,2\}\text{ or }\{1,3,3\}\mod{4}$;
	\item[(iii)] $\{D_1,D_2,D_3\}\equiv\{1,1,1\}\mod{4}$.
\end{itemize}
In cases (i) and (ii), we choose to order the $D_i$ such that $D_1\equiv D_2\mod{4}$ and $|D_1|\leq|D_2|$. For case (iii), take $|D_1|<|D_2|<|D_3|$ and let $\epsilon\in\{\pm1\}$ be such that $\epsilon\equiv g_k\mod{4}$ (this is independent of $k$).
\begin{theorem}[\cite{WilliamsBiquad}]\label{thm:Williams}
If $K$ is a $V_4$-quartic field with quadratic subfields $\QQ(\sqrt{D_i})$ with $D_i$ squarefree, then we have the following cases for the discriminant and integral basis of $K$:
\begin{itemize}
	\item[(i)] $\Delta_K=2^6\cdot(g_1g_2g_3)^2$, basis: $\ds\left(1,\sqrt{D_1},\sqrt{D_3},\frac{\sqrt{D_1}+\sqrt{D_2}}{2}\right)$;
	\item[(ii)] $\Delta_K=2^4\cdot(g_1g_2g_3)^2$, basis: $\ds\left(1,\sqrt{D_1},\frac{1+\sqrt{D_3}}{2},\frac{\sqrt{D_1}+\sqrt{D_2}}{2}\right)$;
	\item[(iii)] $\Delta_K=(g_1g_2g_3)^2$, basis: $\ds\left(1,\frac{1+\sqrt{D_1}}{2},\frac{1+\sqrt{D_2}}{2},\frac{1+\epsilon\sqrt{D_1}+\sqrt{D_2}+\sqrt{D_3}}{4}\right)$.
\end{itemize}
Note that
\[
	 (g_1g_2g_3)^2=D_1D_2D_3.
\]
\end{theorem}

Let $\Delta_i$ be the discriminant of $\QQ(\sqrt{D_i})$. In this section, we will prove the following complete characterization of the shapes of $V_4$-quartic fields.
\begin{theorem}\label{thm:V4_main_theorem}
	The shapes of $V_4$-quartic fields come in two families depending on whether or not $2$ is ramified in $K$ (i.e.\ depending on whether or not $K$ is wild).
	\begin{enumerate}
		\item If $2$ ramifies in $K$, then the combinatorial type of the shape of $K$ is a hexagonal prism (IV) or a cuboid (V). Specifically, the shape is a base-centered orthorhombic lattice ($oC$); in the special case where $D_2=3D_1$ this is a primitive hexagonal lattice ($hP$) and when $D_2=-D_1$ this is a primitive tetragonal lattice (which is the cuboid case). The side ratios of the rectangular prism are $a:b:c=\sqrt{|\Delta_1|}:\sqrt{|\Delta_2|}:\sqrt{|\Delta_3|}$. The shape is primitive hexagonal if and only if all quadratic subfields of $K$ are ramified at $2$ and one of the fields is $\QQ(\sqrt{3})$. The shape is primitive tetragonal if and only if $\QQ(i)$ is a subfield of $K$.
		\item If $2$ is unramified in $K$, then the combinatorial type of the shape of $K$ depends on whether $|D_1|+|D_2|<|D_3|$ or $|D_1|+|D_2|>|D_3|$ (equality cannot occur). In the former case, it is a truncated octahedron (I), while in the latter case it is a rhombo-hexagonal dodecahedron (II). In both cases, the shape is a body-centered orthorhombic lattice ($oI$) with side ratios $a:b:c=\sqrt{|\Delta_1|}:\sqrt{|\Delta_2|}:\sqrt{|\Delta_3|}$, with $a^2+b^2<c^2$ and $a^2+b^2>c^2$, respectively.
	\end{enumerate}
\end{theorem}

\begin{remark}\mbox{}
	\begin{enumerate}
	\item Note that although there are two different combinatorial types when $2$ is unramified, one can deform continuously from one to the other (via the face-centered cubic lattice) as can be seen from the conorm diagrams of Figure~\ref{fig:conorms_oI} (indeed, setting $c^2=a^2+b^2$ in each of the diagrams yields diagrams that are off by an automorphism of the Fano plane).
	\item Similarly, the cuboid combinatorial type when $2$ is ramified is simply a special case of the family of base-centered orthorhombic lattices.
	\item We remark that when the shape is a primitive hexagonal lattice, there can be other $V_4$-quartic fields of the same discriminant that are base-centered orthorhombic lattices. For example, $(D_1,D_2,D_3)=(10,30,3)$ gives a primitive hexagonal lattice, but $(D_1,D_2,D_3)=(2,30,15)$ gives a base-centered orthorhombic.
	\item The shape does not always determine the field. For instance, the two fields with $(D_1,D_2,D_3)$ given by $(-2,-6,3)$ and $(2,6,3)$ have the same shape, as do $(-2,6,-3)$ and $(2,-6,-3)$. We do however have the uniqueness given in Corollary~\ref{cor:V4_shape_complete_invariant} at the end of this section.
	\end{enumerate}
\end{remark}

\subsection{Preliminary calculations}
We collect a few straightforward results used in the following sections. The computations are eased by the fact that $\sqrt{D_i}$ and $\sqrt{D_j}$ ($i\neq j$) are orthogonal, as well as being orthogonal to $1$.

For concreteness (though it doesn't really matter), if $D_i>0$, we let $\sqrt{D_i}$ denote the positive square root of $D_i$, and if $D_i<0$, $\sqrt{D_i}$ will denote its square root whose imaginary part is positive. We fix a choice of orderings of the embeddings of $K$ into $\CC$ such that
\begin{align*}
	j(\sqrt{D_1})&=\left(\sqrt{D_1},-\sqrt{D_1},\sqrt{D_1},-\sqrt{D_1}\right),\\
	j(\sqrt{D_2})&=\left(\sqrt{D_2},\sqrt{D_2},-\sqrt{D_2},-\sqrt{D_2}\right),\\
	j(\sqrt{D_3})&=\left(\sqrt{D_3},-\sqrt{D_3},-\sqrt{D_3},\sqrt{D_3}\right).
\end{align*}
\begin{lemma}\label{lem:usefulcomputations}
	For $1\leq i,k\leq3$,
	\begin{equation}
		\left\langle j(\sqrt{D_i}),j(\sqrt{D_k})\right\rangle=4|D_i|\delta_{ik}.
	\end{equation}
	Furthermore,
	\begin{equation}
		\left\langle j(1),j(\sqrt{D_i})\right\rangle=0.
	\end{equation}
\end{lemma}

\subsection{$K$ ramified at $2$}
As can be seen above, $K$ is ramified at $2$ exactly in cases (i) and (ii).

\subsubsection{Case (i): $(D_1,D_2,D_3)\equiv(2,2,3)\mod{4}$}
\begin{lemma}
	The tuple $\left(\gamma^{(i)}_0,\gamma^{(i)}_1,\gamma^{(i)}_2,\gamma^{(i)}_3\right)$ given by
	\begin{align*}
		\gamma_0^{(i)}&=1-\sqrt{D_1}-\sqrt{D_3}\\
		\gamma_1^{(i)}&=\frac{\sqrt{D}_1-\sqrt{D_2}}{2}\\
		\gamma_2^{(i)}&=\sqrt{D_3}\\
		\gamma_3^{(i)}&=\frac{\sqrt{D}_1+\sqrt{D_2}}{2}.
	\end{align*}
	is an integral basis of $K$.
\end{lemma}
\begin{proof}
	Note that
	\[
		1=\sum_{k=0}^3\gamma_k^{(i)}\quad\text{and}\quad\sqrt{D_1}=\gamma_1^{(i)}+\gamma_3^{(i)}.
	\]
	From this, one may see that the change of basis from the $\gamma_k^{(i)}$ to that of Williams is invertible.
\end{proof}
\begin{proposition}\label{prop:V4_case_i}
	The numbers
	\begin{align*}
	\gamma_{0,\perp}^{(i)}&=-4\left(\sqrt{D_1}+\sqrt{D_3}\right)\\
	\gamma_{1,\perp}^{(i)}&=2\left(\sqrt{D_1}-\sqrt{D_2}\right)\\
	\gamma_{2,\perp}^{(i)}&=4\sqrt{D_3}\\
	\gamma_{3,\perp}^{(i)}&=2\left(\sqrt{D_1}+\sqrt{D_2}\right)
\end{align*}
	form an obtuse superbase of $\mc{O}_K^\perp$. Its Gram matrix (scaled by $2^{-4}$) is
	\begin{equation}
		\vect{4|D_1|+4|D_3|	& -2|D_1|		& -4|D_3|		& -2|D_1|\\
			-2|D_1|			& |D_1|+|D_2|	& 0		& |D_1|-|D_2|\\
			-4|D_3|			& 0		& 4|D_3|	& 0\\
			-2|D_1|		& |D_1|-|D_2|	& 0		& |D_1|+|D_2|}
	\end{equation}
	yielding a conorm diagram as in Figure~\ref{fig:conorms_oC} with $a=2\sqrt{|D_1|}$, $b=2\sqrt{|D_2|}$, and $c=2\sqrt{|D_3|}$. In particular, the shape is a primitive hexagonal lattice if and only if $D_2=3D_1$.
\end{proposition}
\begin{proof}
	That the trace of $\sqrt{D_i}$ is $0$ yields the formulas for the $\gamma_{k,\perp}^{(i)}$. The $\gamma_{k,\perp}^{(i)}$ manifestly form a superbase. One obtains the claimed conorm diagram by simply computing the Gram matrix (using Lemma~\ref{lem:usefulcomputations}).
\end{proof}
From the formula for the discriminant, we see that if the shape is hexagonal, then the discriminant must be divisible by $2^83^2$ in this case. Also, if $D_2=3D_1$, then $D_3=3$. We will need to show that the shape cannot be hexagonal in case (ii).

We also see that $a=b$ if and only if $|D_1|=|D_2|$. Since $D_1\neq D_2$ (or else $D_3=1$), this forces $D_2=-D_1$, in which case $D_3=-1$. Thus, in case (i), a primitive tetragonal lattice occurs only if $\QQ(i)\subseteq K$. We will see that primitive tetragonal lattices cannot occur in case (ii).

\subsubsection{Case (ii): $(D_1,D_2,D_3)\equiv(2,2,1)\text{ or }(3,3,1)\mod{4}$}
\begin{lemma}
	The tuple $\left(\gamma^{(ii)}_0,\gamma^{(ii)}_1,\gamma^{(ii)}_2,\gamma^{(ii)}_3\right)$ given by
	\begin{align*}
		\gamma_0^{(ii)}&=-\sqrt{D_1}+\frac{1-\sqrt{D_3}}{2}\\
		\gamma_1^{(ii)}&=\frac{\sqrt{D}_1-\sqrt{D_2}}{2}\\
		\gamma_2^{(ii)}&=\frac{1+\sqrt{D_3}}{2}\\
		\gamma_3^{(ii)}&=\frac{\sqrt{D}_1+\sqrt{D_2}}{2}.
	\end{align*}
	is an integral basis of $K$.
\end{lemma}
\begin{proof}
	As above, note that
	\[
		1=\sum_{k=0}^3\gamma_k^{(i)}\quad\text{and}\quad\sqrt{D_1}=\gamma_1^{(ii)}+\gamma_3^{(ii)},
	\]
	so that once again the change of basis from the $\gamma_k^{(ii)}$ to that of Williams is invertible.
\end{proof}
\begin{proposition}\label{prop:V4_case_ii}
	The elements
	\begin{align*}
	\gamma_{0,\perp}^{(ii)}&=-2\left(2\sqrt{D_1}+\sqrt{D_3}\right)\\
	\gamma_{1,\perp}^{(ii)}&=2\left(\sqrt{D}_1-\sqrt{D_2}\right)\\
	\gamma_{2,\perp}^{(ii)}&=2\sqrt{D_3}\\
	\gamma_{3,\perp}^{(ii)}&=2\left(\sqrt{D}_1+\sqrt{D_2}\right)
	\end{align*}
	form an obtuse superbase of $\mc{O}_K^\perp$. Its Gram matrix (scaled by $2^{-4}$) is
	\begin{equation}
		\vect{4|D_1|+|D_3|	& -2|D_1|	& -|D_3|		& -2|D_1|\\
			-2|D_1|			& |D_1|+|D_2|	& 0		& |D_1|-|D_2|\\
			-|D_3|			& 0		& |D_3|	& 0\\
			-2|D_1|		& |D_1|-|D_2|	& 0		& |D_1|+|D_2|}
	\end{equation}
	yielding a conorm diagram as in Figure~\ref{fig:conorms_oC} with $a=2\sqrt{|D_1|}$, $b=2\sqrt{|D_2|}$, and $c=\sqrt{|D_3|}$. In particular, the shape, in this case, is never a primitive hexagonal or tetragonal lattice.
\end{proposition}
\begin{proof}
	The proof is along the same lines as for case (i). If $D_2=3D_1$, then $D_3=3$, but ${D_3\equiv1\mod{4}}$, so that the shape is never hexagonal. Similarly, if $a=b$, then $D_2=-D_1$ so that ${D_3=-1\not\equiv1\mod{4}}$.
\end{proof}

\subsection{$K$ unramified at $2$}
Let
\begin{align*}
	\gamma_0&=\frac{1}{4}\left(1+\epsilon\sqrt{D_1}+\sqrt{D_2}+\sqrt{D_3}\right)\\
	\gamma_1&=\frac{1}{4}\left(1-\epsilon\sqrt{D_1}-\sqrt{D_2}+\sqrt{D_3}\right)\\
	\gamma_2&=\frac{1}{4}\left(1+\epsilon\sqrt{D_1}-\sqrt{D_2}-\sqrt{D_3}\right)\\
	\gamma_3&=\frac{1}{4}\left(1-\epsilon\sqrt{D_1}+\sqrt{D_2}-\sqrt{D_3}\right).
\end{align*}

\penalty-100
\begin{proposition}
	The tuple $(\gamma_0,\gamma_1,\gamma_2,\gamma_3)$ is a normal integral basis of $\mc{O}_K$.
\end{proposition}
\begin{proof}
	Note that
	\[
		1=\sum_{i=0}^3\gamma_i,\quad\frac{1+\sqrt{D}_2}{2}=\gamma_0+\gamma_3,\quad\text{and}\quad\frac{1+\sqrt{D_1}}{2}=\begin{cases}
													\gamma_0+\gamma_2,\text{ if }\epsilon=1,\\
													\gamma_1+\gamma_3,\text{ if }\epsilon=-1,
												\end{cases}
	\]
	indicating that the change of basis from the $\gamma_i$ to that of Williams is invertible.
\end{proof}
\begin{proposition}
	When $|D_1|+|D_2|>|D_3|$, the elements
	\begin{align*}
		\gamma_{0,\perp}&=\epsilon\sqrt{D_1}+\sqrt{D_2}+\sqrt{D_3}\\
		\gamma_{1,\perp}&=-\epsilon\sqrt{D_1}-\sqrt{D_2}+\sqrt{D_3}\\
		\gamma_{2,\perp}&=\epsilon\sqrt{D_1}-\sqrt{D_2}-\sqrt{D_3}\\
		\gamma_{3,\perp}&=-\epsilon\sqrt{D_1}+\sqrt{D_2}-\sqrt{D_3}
	\end{align*}
	form an obtuse superbase of $\mc{O}_K^\perp$. Its Gram matrix (scaled by $2^{-2}$) is
	\[
		\vect{|D_1|+|D_2|+|D_3|	& -|D_1|-|D_2|+|D_3|	& |D_1|-|D_2|-|D_3|	& -|D_1|+|D_2|-|D_3|	\\
			-|D_1|-|D_2|+|D_3|	& |D_1|+|D_2|+|D_3|	& -|D_1|+|D_2|-|D_3|	& |D_1|-|D_2|-|D_3|\\
			|D_1|-|D_2|-|D_3|	& -|D_1|+|D_2|-|D_3|	& |D_1|+|D_2|+|D_3|	& -|D_1|-|D_2|+|D_3|	\\
			-|D_1|+|D_2|-|D_3|	& |D_1|-|D_2|-|D_3|	& -|D_1|-|D_2|+|D_3|	&|D_1|+|D_2|+|D_3|
		}
	\]
	yielding a conorm diagram as in Figure~\ref{fig:conorms_oI}(b) with $a=2\sqrt{|D_1|},b=2\sqrt{|D_2|}$, and $c=2\sqrt{|D_3|}$.
\end{proposition}
\begin{proof}
	The proof is similar to previous results. Note in particular that cross terms $\langle\gamma_{i,\perp},\gamma_{k,\perp}\rangle$ ($i\neq k$) vanish, making things simpler. Also, note that combining $|D_1|+|D_2|>|D_3|$ with $|D_1|<|D_2|<|D_3|$ implies that the off-diagonal entries are all negative, as desired.
\end{proof}
When $|D_1|+|D_2|<|D_3|$, we will need a different integral basis for $K$.
\begin{lemma}
	The elements
	\begin{align*}
	\gamma_0^\prime&=\frac{1}{4}\left(1+\epsilon\sqrt{D_1}+\sqrt{D_2}+\sqrt{D_3}\right)\\
	\gamma_1^\prime&=\frac{1}{2}\left(1-\epsilon\sqrt{D_1}\right)\\
	\gamma_2^\prime&=\frac{1}{4}\left(-1+\epsilon\sqrt{D_1}+\sqrt{D_2}-\sqrt{D_3}\right)\\
	\gamma_3^\prime&=\frac{1}{2}\left(1-\sqrt{D_2}\right)
	\end{align*}
	form an integral basis of $\mc{O}_K$.
\end{lemma}
\begin{proof}
	Indeed,
	\[
		\gamma_0^\prime=\gamma_0,\quad\gamma_1^\prime=\gamma_1+\gamma_3,\quad\gamma_2^\prime=-\gamma_1,\quad\text{and}\quad\gamma_3^\prime=\gamma_1+\gamma_2,
	\]
	so that the change of basis between these two collections is invertible.
\end{proof}
\begin{proposition}
	When $|D_1|+|D_2|<|D_3|$, the elements
	\begin{align*}
	\gamma_{0,\perp}^\prime&=\epsilon\sqrt{D_1}+\sqrt{D_2}+\sqrt{D_3}\\
	\gamma_{1,\perp}^\prime&=-2\epsilon\sqrt{D_1}\\
	\gamma_{2,\perp}^\prime&=\epsilon\sqrt{D_1}+\sqrt{D_2}-\sqrt{D_3}\\
	\gamma_{3,\perp}^\prime&=-2\sqrt{D_2}
	\end{align*}
	form an obtuse superbase of $\mc{O}_K^\perp$. Its Gram matrix (scaled by $2^{-2}$) is
	\[
		\vect{
|D_{1}| + |D_{2}| + |D_{3}| & -2 |D_{1}| & |D_{1}| + |D_{2}| - |D_{3}| & -2 |D_{2}| \\
-2 |D_{1}| & 4 |D_{1}| & -2 |D_{1}| & 0 \\
|D_{1}| + |D_{2}| - |D_{3}| & -2 |D_{1}| & |D_{1}| + |D_{2}| + |D_{3}|& -2 |D_{2}| \\
-2 |D_{2}| & 0 & -2|D_{2}| & 4 |D_{2}|
		}
	\]
	yielding a conorm diagram as in Figure~\ref{fig:conorms_oI}(a) with $a=2\sqrt{|D_1|},b=2\sqrt{|D_2|}$, and $c=2\sqrt{|D_3|}$.
\end{proposition}
\begin{proof}
	Similar to above.
\end{proof}
This completes the proof of Theorem~\ref{thm:V4_main_theorem}.

\subsection{Uniqueness of the shape}
Although different $V_4$-quartic fields can have the same shape, we have the following result on the uniqueness of the shape in certain natural families.
\begin{corollary}\label{cor:V4_shape_complete_invariant}\mbox{}
	\begin{enumerate}
		\item The shape of a totally real $V_4$-quartic field determines it amongst the family of all totally real $V_4$-quartic fields.
		\item The shape of a tame $V_4$-quartic field determines it amongst the family of all tame $V_4$-quartic fields.
	\end{enumerate}
\end{corollary}
\begin{proof}
	Suppose you know that you have the shape of a totally real field $K$. Knowing the shape tells you the ratios $\Delta_1:\Delta_2:\Delta_3$. A representative of these ratios is $(1,\Delta_2/\Delta_1,\Delta_3/\Delta_1)$. In cases (i) and (iii),
	\[
		\frac{\Delta_2}{\Delta_1}=\frac{g_1}{g_2}\quad\text{and}\quad\frac{\Delta_3}{\Delta_1}=\frac{g_1}{g_3},
	\]
	since $D_i/D_j=g_j/g_i$. In case case (ii),
	 \[
		\frac{\Delta_2}{\Delta_1}=\frac{g_1}{g_2}\quad\text{and}\quad\frac{\Delta_3}{\Delta_1}=\frac{g_1}{4g_3}
	\]
	In all cases, $2\nmid g_1,g_2$ and the $g_i$ are pairwise relatively prime, so these fractions are in lowest terms. Clearing denominators therefore yields
	\[
		(g_2g_3,g_1g_3,g_1g_2)\quad\text{or}\quad(4g_2g_3,4g_1g_3,g_1g_2),
	\]
	respectively. If the first two entries of the tuple you obtain from clearing denominators are $0$ modulo $4$, you then know you are in case (ii) and the tuple gives you the three discriminants $\Delta_1,\Delta_2,\Delta_3$, thus telling you the quartic field. If the three entries are $1$ modulo $4$, you know you are in case (iii) and once again the tuple is telling you the three discriminants of the quadratic subfields of $K$. Otherwise, you must be in case (i) and you get the three discriminants by multiplying the tuple by $4$.
	
	Suppose now that you know you have the shape of a field $K$ in which $2$ is unramified (equivalently $K$ is tamely ramified). Similarly, you can get the triple $(1,\left|\Delta_2/\Delta_1\right|,\left|\Delta_3/\Delta_1\right|)$. Clearing	denominators gives $(|g_2g_3|,|g_1g_3|,|g_1g_2|)$. If all these entries are $1$ modulo $4$, then you know you have a totally real field and the tuple is telling you the three discriminants $|D_i|$. Otherwise, two of the entries must be $3$ modulo $4$. Flipping the signs on these then gives the three discriminants of the quadratic subfields of $K$, once again telling you which field $K$ is.
\end{proof}

\section{The equidistribution of shapes of $V_4$-quartic fields}\label{sec:V4equid}

In this section, we prove Theorem~\ref{thm:V4equid} that the shapes of $V_4$-quartic fields are equidistributed (in a regularized sense) in appropriate two-dimensional spaces. To accomplish this, we use the Principle of Lipschitz and a fairly straightforward sieve. The result reduces to counting strongly carefree triples in a certain region of space and satisfying certain congruence conditions. This counting is done in \S\ref{sec:carefreecounting}. We begin by making explicit the relation between the fields we want to count and asymptotics for strongly carefree triples.

\subsection{Reduction to counting strongly carefree triples}
We break up the set of $V_4$-quartic fields according to the cases (i)--(iii) of \S\ref{sec:V4_shapes}. For $?=(i),(ii),$ or $(iii)$, let $\mc{K}^?$ denote the set of $V_4$-quartic fields that are in case ?. As described at the beginning of \S\ref{sec:V4_shapes}, a $V_4$-quartic field $K$ is determined by its three quadratic subfields $\QQ(\sqrt{D_1}),\QQ(\sqrt{D_2}),\QQ(\sqrt{D_3})$. Let
\[
	\mc{D}:=\left\{(D_1,D_2,D_3)\in\ZZ^3:D_i\neq0,1\text{ is squarefree and for }\{i,j,k\}=\{1,2,3\}, D_i=\dfrac{D_jD_k}{\gcd(D_j,D_k)^2}\right\}
\]
and
\begin{align*}
	\mc{D}^{(i)}	&:=\left\{(D_1,D_2,D_3)\in\mc{D}:D_1\equiv D_2\equiv2\mod{4}, D_3\equiv3\mod{4},|D_1|\leq|D_2|\right\},\\
	\mc{D}^{(ii)}	&:=\left\{(D_1,D_2,D_3)\in\mc{D}:D_1\equiv D_2\equiv2\mod{4}, D_3\equiv1\mod{4},|D_1|<|D_2|\right\},\\
				&\phantom{:=}\cup\left\{(D_1,D_2,D_3)\in\mc{D}:D_1\equiv D_2\equiv3\mod{4}, D_3\equiv1\mod{4},|D_1|<|D_2|\right\},\\
	\mc{D}^{(iii)}	&:=\left\{(D_1,D_2,D_3)\in\mc{D}:D_i\equiv1\mod{4}\text{ for each }i,|D_1|<|D_2|<|D_3|\right\}.\\
\end{align*}
We then have bijections between $\mc{K}^?$ and $\mc{D}^?$ for each of $?=(i),(ii),(iii)$. It will be convenient for counting purposes to replace the triples in $\mc{D}$ with triples of their gcd's. We will, in fact, slightly modify the notion of gcd when negative numbers are involved, essentially considering $-1$ as a prime.
\begin{definition}\mbox{}
	\begin{enumerate}
		\item For positive integers $a$ and $b$, we define
		\begin{align*}
			\ga(a,b)&:=\ga(-a,b):=\ga(a,-b):=\gcd(a,b)\\
			\ga(-a,-b)&:=-\gcd(a,b).\\
		\end{align*}
		We say that two integers $a$ and $b$ are $\ast$-relatively prime if $\ga(a,b)=1$. In particular, two negative integers are never $\ast$-relatively prime.
		\item A \textit{$\ast$-strongly carefree triple} is $(g_1,g_2,g_3)\in\ZZ^3$ such that the $g_i$ are squarefree, distinct, and pairwise $\ast$-relatively prime.
	\end{enumerate}
\end{definition}

Let $\mc{SC}$ denote the set of $\ast$-strongly carefree triples. Then the map
\[
	(g_1,g_2,g_3)\mapsto(g_2g_3,g_1g_3,g_1g_2)
\]
gives a bijection from $\mc{SC}$ to $\mc{D}$ with inverse
\[
	(D_1,D_2,D_3)\mapsto(\ga(D_2,D_3),\ga(D_1,D_3),\ga(D_1,D_2)).
\]
For $?=(i),(ii),(iii)$, let
\begin{align*}
	\mc{SC}^?&:=\left\{(g_1,g_2,g_3)\in\mc{SC}:(g_2g_3,g_1g_3,g_1g_2)\in\mc{D}^?\right\}.
\end{align*}
The above bijection restricts to bijections between $\mc{SC}^?$ and $\mc{D}^?$.

To incorporate a discriminant bound, for a positive real number $X$, let
\begin{align*}
	X_{(i)}&=\frac{X}{2^6},\\
	X_{(ii)}&=\frac{X}{2^4},\\
	X_{(iii)}&=X,
\end{align*}
and, for $?=(i),(ii),(iii)$, let
\begin{align*}
	\mc{D}^?(X_?)&:=\left\{(D_1,D_2,D_3)\in\mc{D}^?:D_1D_2D_3<X_?\right\},\\
	\mc{SC}^?(X_?)&:=\left\{(g_1,g_2,g_3)\in\mc{SC}^?:(g_1g_2g_3)^2<X_?\right\}.\\
\end{align*}
It then follows from Theorem~\ref{thm:Williams} that the bijections between $\mc{K}^?,\mc{D}^?$, and $\mc{SC}^?$ restrict to bijections between $\mc{K}^?(X_?),\mc{D}^?(X_?)$, and $\mc{SC}^?(X_?)$.

Finally, we must select for the shapes of the fields we are counting. Note that for $i\neq j$,
\[
	\frac{D_i}{D_j}=\frac{g_j}{g_i}
\]
and $|D_i|\leq |D_j|$ if and only if $|g_i|\geq |g_j|$. Let $(D_1,D_2,D_3)\in\mc{D}^{(i)}$ and let $K$ be the corresponding field. We have that $\Delta_i=4D_i$, so that by Theorem~\ref{thm:V4_main_theorem}, the shape of $K$ is $\sh_{oC}(x,y)$ with
\[
	x=\sqrt{\left|\frac{D_1}{D_3}\right|}=\sqrt{\left|\frac{g_3}{g_1}\right|}\leq y=\sqrt{\left|\frac{D_2}{D_3}\right|}=\sqrt{\left|\frac{g_3}{g_2}\right|}.
\]
For $(D_1,D_2,D_3)\in\mc{D}^{(ii)}$, we have that $\Delta_i=4D_i$ for $i=1,2$ and $\Delta_3=D_3$, so that the shape of the corresponding field is $\sh_{oC}(x,y)$ with
\[
	x=2\sqrt{\left|\frac{D_1}{D_3}\right|}=2\sqrt{\left|\frac{g_3}{g_1}\right|}<y=2\sqrt{\left|\frac{D_2}{D_3}\right|}=2\sqrt{\left|\frac{g_3}{g_2}\right|}.
\]
Finally, for $(D_1,D_2,D_3)\in\mc{D}^{(iii)}$, $\Delta_i=D_i$, so that the shape of the corresponding field is $\sh_{oI}(x,y)$ with
\[
	x=\sqrt{\left|\frac{D_1}{D_3}\right|}=\sqrt{\left|\frac{g_3}{g_1}\right|}< y=\sqrt{\left|\frac{D_2}{D_3}\right|}=\sqrt{\left|\frac{g_3}{g_2}\right|}.
\]
Let $s_{(ii)}=2$ and $s_{(i)}=s_{(ii)}=1$, and for two positive real numbers $R_1<R_2$, define
\begin{align*}
	\mc{K}^?(X_?,R_1,R_2)&:=\left\{K\in\mc{K}^?(X_?):\sh(K)\in W_?(R_1,R_2)\right\},\\
	\mc{D}^?(X_?,R_1,R_2)&:=\left\{(D_1,D_2,D_3)\in\mc{D}^?(X_?):R_1^2\leq s_?^2\left|D_1/D_3\right|\leq s_?^2\left|D_2/D_3\right|<R_2^2\right\},\\
	\mc{SC}^?(X_?,R_1,R_2)&:=\left\{(g_1,g_2,g_3)\in\mc{SC}^?:R_1^2\leq s_?^2\left|g_3/g_1\right|\leq s_?^2\left|g_3/g_2\right|<R_2^2\right\},\\
\end{align*}
where $W_?$ refers to $W_{oC}$ for $?=(i),(ii)$ and $W_{(iii)}=W_{oI}$. We have shown that
\begin{proposition}\label{prop:bijectionsKDSC}
	The bijections between $\mc{K}^?(X_?),\mc{D}^?(X_?)$, and $\mc{SC}^?(X_?)$ restrict to bijections between $\mc{K}^?(X_?,R_1,R_2),\mc{D}^?(X_?,R_1,R_2)$, and $\mc{SC}^?(X_?,R_1,R_2)$.
\end{proposition}

We have thus translated our problem of counting $V_4$-quartic fields with bounded discriminant and shape in some ``box'' into a problem of counting $\ast$-strongly carefree triples satisfying certain congruence conditions lying in some region.

\subsection{Counting $\ast$-strongly carefree triples with congruence conditions}\label{sec:carefreecounting}
Our strategy for counting elements of $\mc{SC}^?(X_?,R_1,R_2)$ will be to first count triples of integers satisfying finitely many of the correct congruences, then to apply a sieve to get a count of $\ast$-strongly carefree triples.

In the previous section, we set up a bijection between $V_4$-quartic fields with bounded discriminant and constrained shape and certain triples of integers. We will view these triples as lattice points in a region of $\RR^3$ and use the Principle of Lipschitz to estimate the number of them. The Principle of Lipschitz basically estimates the number of lattices points in a ``nice'' region as the volume of that region with an error given by the lower-dimensional volumes of the projections of the region onto coordinate hyperplanes (see e.g.\ \cite[Lemma~9]{ManulQuartic} for a precise statement). Accordingly, for $N,r_1,r_2>0$ with $r_1<r_2$, let
\[
	\mc{R}(N,r_1,r_2):=\left\{(g_1,g_2,g_3)\in\RR^{\times3}:|g_1g_2g_3|<N,r_1\leq|g_3/g_1|\leq|g_3/g_2|<r_2\right\}.
\]
and let $\mc{R}^0(N,r_1,r_2)$ be its intersection with the octant $x_i>0$.
\begin{lemma}
	The volume of $\mc{R}(N,r_1,r_2)$ is
	\[
		\frac{4N}{3}\left(\log(r_2)-\log(r_1)\right)^2
	\]
	and the maximum measure of this region's lower-dimensional shadows on coordinate hyperplanes is $O(N^{2/3})$.
\end{lemma}
\begin{proof}
	First note that the volume of $\mc{R}$ is $8$ times that of $\mc{R}^0$ and the measures of the shadows are at most $4$ times those of $\mc{R}^0$. We therefore consider $\mc{R}^0$. We make the change of variables
	\begin{align*}
		x_1&=\frac{g_3}{g_1},\\[5pt]
		x_2&=\frac{g_3}{g_2},\\[5pt]
		x_3&=(g_1g_2)^3.
	\end{align*}
	The Jacobian determinant of this change of variables is
	\[
		\begin{vmatrix}
			-g_3g_1^{-2}&0&g_1^{-1}\\
			0&-g_3g_2^{-2}&g_2^{-1}\\
			3g_1^2g_2^3&3g_1^3g_2^2&0
		\end{vmatrix}
		=6g_1g_2g_3=6\sqrt{x_1x_2x_3}.
	\]
	Therefore,
	\begin{align*}
		\int_{\mc{R}(N,r_1,r_2)}dg_1dg_2dg_3	&=\int_{r_1}^{r_2}\int_{r_1}^{x_2}\int_0^{N^2/x_1x_2}\dfrac{1}{6\sqrt{x_1x_2x_3}}dx_3dx_1dx_2\\
											&=2\cdot\frac{1}{6}\int_{r_1}^{r_2}\int_{r_1}^{x_2}\frac{1}{\sqrt{x_1x_2}}\cdot\frac{N}{\sqrt{x_1x_2}}dx_1dx_2\\
											&=\frac{N}{3}\int_{r_1}^{r_2}\int_{r_1}^{x_2}\frac{1}{x_1x_2}dx_1dx_2.
	\end{align*}
	This latter integral is just like the one in \eqref{eqn:WoCintegralFirst}, yielding the claimed value.

	To bound the measures of the shadows, we will simply show that $g_i=O(N^{1/3})$ for $i=1,2,3$; the shadows will then be contained inside boxes of side $O(N^{1/3})$ of dimension at most $2$. Note that $r_1g_1\leq g_3$ and $g_3/r_2\leq g_2$. Thus,
	\begin{align*}
		N	&>g_1g_2g_3\\
			&\geq\frac{1}{r_2}g_1g_3^2\\
			&\geq\frac{r_1^2}{r_2}g_1^3,
	\end{align*}
	so that
	\[
		g_1<\left(\frac{r_2}{r_1^2}\right)^{1/3}N^{1/3},
	\]
	as desired. Proceeding similarly, we obtain
	\[
		g_2<\left(\frac{r_1}{r_2^2}\right)^{1/3}N^{1/3}\quad\text{and}\quad g_3<r_2^{2/3}N^{1/3}.
	\]
\end{proof}

For a subset $\mc{L}\subseteq\ZZ^3$, let
\[
\mc{R}_{\mc{L}}(N,r_1,r_2):=\mc{L}\cap\mc{R}(N,r_1,r_2).
\]
Applying the Principle of Lipschitz, we get the following count of all lattice points in the above region.
\begin{corollary}
	For $N,r_1,r_2>0$ with $r_1<r_2$,
	\[
		\#\mc{R}_{\ZZ^3}(N,r_1,r_2)=\frac{4N}{3}\left(\log(r_2)-\log(r_1)\right)^2+O(N^{2/3}).
	\]
\end{corollary}

We now generalize this result to include finitely many congruences conditions.
\begin{definition}
	Let $n\in\ZZ_{\geq1}$.
	\begin{enumerate}
		\item We say that two integers $a$ and $b$ are \textit{congruent modulo }$n(\infty)$ if they are congruent modulo $n$ and have the same sign.
		\item By a \textit{set of congruence conditions modulo }$n(\infty)$, we mean a subset $\mc{C}$ of $\left(\{\pm\}\times\ZZ/n\ZZ\right)^3$.
		\item We will say that $(g_1,g_2,g_3)\in\ZZ^3$ is \textit{in} $\mc{C}$ if
		\[
			((\sgn g_1,g_1+n\ZZ),(\sgn g_2,g_2+n\ZZ),(\sgn g_3,g_3+n\ZZ))\in\mc{C}.
		\]
	\end{enumerate}
\end{definition}
Let $\mc{C}$ be a set of congruence conditions modulo $n(\infty)$. We will be interested in sets of the form
\[
	\mc{L}_\mc{C}=\left\{(g_1,g_2,g_3)\in\ZZ^3:(g_1,g_2,g_3)\text{ is \textit{not} in }\mc{C}\right\}.
\]
By the Chinese Remainder Theorem, we may split up the congruence conditions into prime powers. For each prime number $p$, let $v_p(n)$ denote the biggest power of $p$ dividing $n$. Given $\mc{C}$, there is a set of congruence conditions $\mc{C}_p\subseteq(\ZZ/p^{v_p(n)}\ZZ)^3$ and a $\mc{C}_\infty\subseteq\{\pm\}^3$ such that
\[
	(g_1,g_2,g_3)\text{ is in }\mc{C}\quad\text{if and only if}\quad(g_1,g_2,g_3)\text{ is in }\mc{C}_p\text{ for all }p\leq\infty.\footnote{For $p=\infty$, we mean that the signs of $g_1,g_2,g_3$ are in $\mc{C}_\infty$.}
\]
For a prime $p$, we call the \textit{$p$-adic density of $\mc{L}_{\mc{C}}$} the rational number
\[
	\delta_p(\mc{L}_\mc{C}):=1-\frac{\#\mc{C}_p}{p^{3v_p(n))}}.
\]
For $p=\infty$, let
\[
	\delta_\infty(\mc{L}_\mc{C}):=1-\frac{\#\mc{C}_\infty}{2^3}.
\]
\begin{proposition}\label{prop:arbitrarycongruenceconditions}
	Fix a set of congruence conditions $\mc{C}$ modulo $n(\infty)$. Then,
	\[
		\mc{R}_{\mc{L}_\mc{C}}(N,r_1,r_2)=\frac{4}{3}\left(\prod_{p\leq\infty}\delta_p(\mc{L}_\mc{C})\right)N\left(\log(r_2)-\log(r_1)\right)^2+O(N^{2/3}).
	\]
\end{proposition}
\begin{proof}
For $\ul{m}=((\sigma_1,m_1),(\sigma_2,m_2),(\sigma_3,m_3))\in\left(\{\pm\}\times\ZZ/n\ZZ\right)^3$, let
\[
	\L_{\ul{m}}=((m_1,m_2,m_3)+n\ZZ^3)\cap\RR^3_{\sigma_1,\sigma_2,\sigma_3},
\]
where $\RR^3_{\sigma_1,\sigma_2,\sigma_3}$ denotes the octant in $\RR^3$ given by the signs $(\sigma_1,\sigma_2,\sigma_3)$.
We can write
	\[
		\mc{L}_\mc{C}=\bigcup_{\ul{m}\notin\C}\L_{\ul{m}},
	\]
	i.e.\ $\L$ is a union of translates of scalings of $\ZZ^3$ (with certain restrictions to octants). The Principle of Lipschitz applies to each $\mc{L}_{\ul{m}}$ though we must scale by $n$. We obtain that	\[
		\#\mc{R}_{\L_{\ul{m}}}(N,r_1,r_2)=\frac{4}{3}(2n)^{-3}N\left(\log(r_2)-\log(r_1)\right)^2+O(N^{2/3}).
	\]
	Summing over $\ul{m}\notin\C$ yields the desired result since
	\[
		\#\C^c\cdot (2n)^3=\prod_{p\leq\infty}\delta_p(\L_\C),
	\]
	where $\C^c$ denotes the complement of $\C$.
\end{proof}

In order to count strongly carefree triples, we must impose the following \textit{infinitely many} congruence conditions: for all primes $p$,

\penalty-100
\begin{itemize}
	\item there is no $i$ such that $g_i\equiv0\mod{p^2}$ (squarefree),
	\item if $g_i\equiv0\mod{p}$, then there is no $j\neq i$ such that $g_j\equiv0\mod{p}$ (pairwise relatively prime).
\end{itemize}
Accordingly, define the congruence condition $\C_p^{\sf}$ modulo $p^2$ by
\[
	\C_p^{\sf}:=\left\{(g_1,g_2,g_3)\in(\ZZ/p^2\ZZ)^3:\text{at least two of the }g_i\text{ are 0 modulo }p\right\}.
\]
A strongly carefree triple is $\ast$-strongly carefree if and only if at most $1$ of the numbers is negative, so we define $\C_\infty^{\sf}$ to be the triples of signs at least two of which are negative. The following lemma will be needed in applying a sieve below.
\begin{lemma}\label{lem:sizeofCp}
	For a prime $p$,
	\[
		\#\C_p^{\sf}=6p^4-8p^3+3p^2.
	\]
\end{lemma}
\begin{proof}
	First consider the tuples $(g_1,g_2,g_3)\in(\ZZ/p^2\ZZ)^3$ at least one of whose coordinates is $0$. The three ``coordinate planes'' each have $(p^2)^2$. Each pair of them intersects in a ``coordinate axis'', each having $p^2$ points. The intersection of all three planes is the origin. Therefore, inclusion-exclusion yields that there are $3p^4-3p^2+1$ such tuples.
	
	Now, consider the tuples none of whose coordinates is $0$. For $\{i,j,k\}=\{1,2,3\}$, let
	\[
		\C_{p,k}^{\sf}:=\{(g_1,g_2,g_3)\in(\ZZ/p^2\ZZ)^3:g_i,g_j\equiv0\mod{p},g_i,g_j,g_k\neq0\}.
	\]
	For each of $g_i$ and $g_j$, there are $p-1$ values that are $0\mod{p}$ (but not modulo $p^2$). And for each pair of such values, every $p^2-1$ non-zero value of $g_k$ yields a tuple in $\C_{p,k}^{\sf}$, so that
	\[
		\#\C_{p,k}^{\sf}=(p-1)^2(p^2-1).
	\]
	The intersection of any two $\C_{p,k}^{\sf}$ (or all three) consists of tuples all of whose coordinates are $0\mod{p}$, of which there are $(p-1)^3$. Inclusion-exclusion then says that the number of tuples in $\C_p^{\sf}$ under consideration is
	\[
		3(p-1)^2(p^2-1)-3(p-1)^3+(p-1)^3=3p^4-8p^3+6p^2-1.
	\]
	Combining the two pieces of $\C_p^{\sf}$ yields the result.
\end{proof}

In addition, to being $\ast$-strongly carefree, the triples we are interested in must satisfy certain congruences modulo $4$ that ensure they correspond to conditions $(i),(ii),$ or $(iii)$, respectively. Define the following subsets of $(\ZZ/4\ZZ)^3$:
\begin{align*}
	\C_2^{(i)}&:=(\ZZ/4\ZZ)^3\setminus\{(1,3,2),(3,1,2)\},\\
	\C_2^{(ii)}&:=(\ZZ/4\ZZ)^3\setminus\{(1,1,2),(3,3,2),(1,1,3),(3,3,1)\},\\
	\C_2^{(iii)}&:=(\ZZ/4\ZZ)^3\setminus\{(1,1,1),(3,3,3)\}.
\end{align*}
\begin{lemma}
	For $?=(i),(ii),(iii)$, a $\ast$-strongly carefree triple $(g_1,g_2,g_3)$ is in $\mc{SC}^?$ if and only if it is not in $\C_2^?$.
\end{lemma}
\begin{proof}
	For case $(i)$, we must have that $(g_2g_3,g_1g_3,g_1g_2)\equiv(2,2,3)\mod{4}$. Therefore $g_1$ and $g_2$ must be odd and not congruent modulo $4$. This forces $g_3$ to be $2$ modulo $4$, so that $(g_1,g_2,g_3)\equiv(1,3,2)$ or $(3,1,2)\mod{4}$. The other cases are similar.
\end{proof}

Since $\C_2^?\supseteq\C_2^{\sf}$, these conditions at $2$ already take care of the strongly carefree condition with respect to the prime $2$. Accordingly, for $Y\geq2$ and for $?=(i),(ii),(iii)$, let
\[
	n(Y):=\prod_{p\leq Y}p^2,
\]
and let $\C^?_Y$ denote the set of congruence conditions modulo $n(Y)(\infty)$ given by $\C_\infty^{\sf}$, $\C_2^?$, and $\C_p^{\sf}$ for $2<p\leq Y$. Let $\mc{L}^?(Y)=\mc{L}_{\C^?_Y}$. By Lemma~\ref{lem:sizeofCp}, for $2<p\leq Y$,
\[
	\delta_p(\mc{L}^?(Y))=1-6p^{-2}+8p^{-3}-3p^{-4}.
\]
We also have that $\delta_\infty(\mc{L}^?(Y))=1/2$ and $\delta_2(\mc{L}^?(Y))=s_?/32$. Applying Proposition~\ref{prop:arbitrarycongruenceconditions} to $\L^?(Y)$, we obtain the following intermediary result.
\begin{corollary}
	For $0<r_1<r_2$, we have that
	\[
		\mc{R}_{\mc{L}^?(Y)}(N,r_1,r_2)=\frac{s_?}{48}\prod_{2<p\leq Y}\left(1-6p^{-2}+8p^{-3}-3p^{-4}\right)N\left(\log(r_2)-\log(r_1)\right)^2+O(N^{2/3}).
	\]
\end{corollary}

We must now show that we can take the limit as $Y\rightarrow\infty$ above and obtain the same asymptotic. We accomplish this with a sieve adapted from \cite[\S5]{DH}. This method worsens the error to $o(N)$, but that is sufficient for our purposes. Let $\L^?_\infty$ be the set where the congruence conditions modulo all primes are imposed. We have that
\begin{align}
	\limsup_{N\rightarrow\infty}\frac{\#\mc{R}_{\L^?_\infty}(N,r_1,r_2)}{N}&\leq \lim_{Y\rightarrow\infty}\lim_{N\rightarrow\infty}\frac{\#\mc{R}^?_{\L^?(Y)}(N,r_1,r_2)}{N}\nonumber\\
		&\leq\frac{s_?}{48}\prod_{p\text{ odd}}\left(1-6p^{-2}+8p^{-3}-3p^{-4}\right)\left(\log(r_2)-\log(r_1)\right)^2. \label{eqn:sieveupperbound}
\end{align}

Let
\[
	\W^?_p:=\{(g_1,g_2,g_3)\in\ZZ^3:(g_1,g_2,g_3)\text{ is in }\C^?_p\}.
\]
Then
\[
	\mc{R}_{\L^?(Y)}(N,r_1,r_2)\subseteq\mc{R}_{\L^?_\infty}(N,r_1,r_2)\cup\bigcup_{p>Y}\mc{R}_{\W_p}(N,r_1,r_2).
\]
Thus,
\begin{equation}\label{eqn:sievelowerbound}
	\frac{\#\mc{R}_{\L^?_\infty}(N,r_1,r_2)}{N}\geq\frac{\#\mc{R}_{\L^?(Y)}(N,r_1,r_2)}{N}-O\left(\sum_{p>Y}\frac{\#\mc{R}_{\W^?_p}(N,r_1,r_2)}{N}\right).
\end{equation}
By Lemma~\ref{lem:sizeofCp},
\[
	\frac{\#\mc{R}_{\W^?_p}(N,r_1,r_2)}{N}=O(p^{-2}),
\]
so that the sum in the big-oh goes to zero as $Y$ goes to infinity. Taking \eqref{eqn:sieveupperbound} with the liminf of \eqref{eqn:sievelowerbound} as $N\rightarrow\infty$ and taking the limit as $Y$ approaches $\infty$ yields
\[
\#\mc{R}_{\L^?_\infty}(N,r_1,r_2)=\frac{s_?}{48}\prod_{p\text{ odd}}\left(1-6p^{-2}+8p^{-3}-3p^{-4}\right)N\left(\log(r_2)-\log(r_1)\right)^2+o(N).
\]

We now put this all together. Recall from Theorem~\ref{thm:V4TameWild} that the shape of a $V_4$-quartic field $K$ lies in one of two spaces, $\mc{S}_{oC}$ or $\mc{S}_{oI}$, depending on whether $2$ is ramified in $K$ or not. The following result thus breaks up into these two cases.
\begin{theorem}
	Let $0<R_1<R_2$.
	\begin{enumerate}
		\item The number of $V_4$-quartic fields $K$ in which $2$ is ramified, $\Delta_K<X$, and $\sh(K)\in W_{oC}(R_1,R_2)$ is
		\[
			\frac{5}{48}\prod_{p\text{ odd}}\left(1-6p^{-2}+8p^{-3}-3p^{-4}\right)X^{1/2}\mu_{oC}(W_{oC}(R_1,R_2))+o(X^{1/2}).
		\]
		\item Assume further that $R_2<1$. The number of $V_4$-quartic fields $K$ in which $2$ is unramified, $\Delta_K<X$, and $\sh(K)\in W_{oI}(R_1,R_2)$ is
		\[
			\frac{1}{6}\prod_{p\text{ odd}}\left(1-6p^{-2}+8p^{-3}-3p^{-4}\right)X^{1/2}\mu_{oI}(W_{oI}(R_1,R_2))+o(X^{1/2}).
		\]
	\end{enumerate}
\end{theorem}
\begin{proof}
	For $?=(i),(ii),(ii)$, we have that
	\[
		\mc{SC}^?(X_?,R_1,R_2)=\mc{R}_{\L^?_\infty}\left(X_?^{1/2},\left(\frac{R_1}{s_?}\right)^2,\left(\frac{R_2}{s_?}\right)^2\right).
	\]
	Let us first deal with the wild case. Proposition~\ref{prop:bijectionsKDSC} tells use that the number we seek is
	\[
		\mc{SC}^{(i)}(X_{(i)},R_1,R_2)+\mc{SC}^{(ii)}(X_{(ii)},R_1,R_2).
	\]
	We have that
	\[
		\mc{R}_{\L^{(i)}_\infty}\left(\frac{X^{1/2}}{2^3},R_1^2,R_2^2\right)=\frac{1}{96}\prod_{p\text{ odd}}\left(1-6p^{-2}+8p^{-3}-3p^{-4}\right)X^{1/2}\left(\log(R_2)-\log(R_1)\right)^2+o(X^{1/2})
	\]
	and
	\[
		\mc{R}_{\L^{(ii)}_\infty}\left(\frac{X^{1/2}}{2^2},\frac{1}{4}R_1^2,\frac{1}{4}R_2^2\right)=\frac{1}{24}\prod_{p\text{ odd}}\left(1-6p^{-2}+8p^{-3}-3p^{-4}\right)X^{1/2}\left(\log(R_2)-\log(R_1)\right)^2+o(X^{1/2}).
	\]
	
	In the tame case, we have that
	\[
		\mc{R}_{\L^{(iii)}_\infty}\left(X^{1/2},R_1^2,R_2^2\right)=\frac{1}{12}\prod_{p\text{ odd}}\left(1-6p^{-2}+8p^{-3}-3p^{-4}\right)X^{1/2}\left(\log(R_2)-\log(R_1)\right)^2+o(X^{1/2})
	\]
\end{proof}

By Lemma~\ref{lem:WoCsuffice} and its analogue for $\mc{S}_{oI}$, this proves Theorem~\ref{thm:V4equid}.


\section{The shapes of $C_4$-quartic fields}\label{sec:C4fields_description}

Let us begin by stating the main theorem of this section. To this end, we first note that the discriminant of every $C_4$-quartic field $K$ is of the form $2^eA^2D^3$ with $A$ odd and squarefree, and $D$ relatively prime to $A$ and squarefree. Then, $D$ is the product of all primes that ramify in the unique quadratic subfield $K_2$ of $K$ and $A$ is the product of all odd primes that ramify in $K$, but not in $K_2$; we take $A<0$ when $K$ is not totally real. Let $\mc{N}=\mc{N}_K$ denote the absolute norm of the relative discriminant of $K/K_2$ and let $\Delta_2$ denote the discriminant of $K_2$. 

The goal of this section is to prove the following.
\begin{theorem}\label{thm:C4_main_theorem}
	The shapes of $C_4$-quartic fields $K$ come in two families depending on whether or not $K$ is wildly ramified (i.e.\ whether or not $2$ is ramified in $K$).
	\begin{enumerate}
		\item If $2$ is unramified in $K$, then the combinatorial type of the shape of $K$ is a truncated octahedron. Specifically, the shape is a body-centered tetragonal lattice ($tI$) whose side ratio is ${\left(\dfrac{|\Delta_2|}{\mc{N}}\right)^{1/4}\leq1}$. When this ratio is $1$, this is a body-centered cubic lattice ($cI$), and this occurs if and only if $\Delta_K$ is a cube, i.e.\ if and only if no new primes ramify in $K/K_2$.
		\item If $2$ ramifies in $K$, then the combinatorial type of the shape of $K$ is a cuboid. Specifically, the shape is a primitive tetragonal lattice ($tP$) whose side ratio is $\left(\dfrac{4|\Delta_2|}{\mc{N}}\right)^{1/4}\leq1$. The shape is a primitive cubic lattice ($cP$) if and only if $\Delta_K=2^{11}\delta$, where $\delta$ is an odd cube, i.e.\ if and only if $2$ ramifies in $K_2$ and no new primes ramify in $K/K_2$.
	\end{enumerate}
\end{theorem}
Along the way we will prove several more explicit results that are also of interest (e.g.~Lemma~\ref{lem:C4_Gamma}, and Propositions~\ref{prop:C4_case_i}, \ref{prop:C4_case_ii}, and \ref{prop:C4_case_iii}). We begin with some remarks.
\begin{remark}\mbox{}
	\begin{enumerate}
		\item There will be five cases we deal with, essentially depending on the ramification of $2$. The ratio $\left(\dfrac{4|\Delta_2|}{\mc{N}}\right)^{1/4}$ is given by $|A|^{-1/2}$ in case (i), and by $(4|A|)^{-1/2}$ and $(2|A|)^{-1/2}$ in cases (ii) and (iii), respectively. In cases (iv) and (v), $\left(\dfrac{|\Delta_2|}{\mc{N}}\right)^{1/4}=|A|^{-1/2}$. See Lemma~\ref{lem:translateABCDND2} below for these formulas.
		\item A simple argument using class field theory shows that if $p$ is an odd prime, then $v_p(\Delta_K)=3$ implies $p\equiv1\mod{4}$. Specifically, only $2$ and primes that are $1\mod{4}$ can be ramified in $K_2$ and all primes that ramify in $K_2$ must also ramify in $K/K_2$.
	\end{enumerate}
\end{remark}

In \cite{HHRWH}, it is shown that every $C_4$-quartic field $K$ can be written uniquely in the form ${K=\QQ(\alpha)}$, where $\alpha=\sqrt{A(D+B\sqrt{D})}$ with $A,B,C,D\in\ZZ$ satisfying
\begin{itemize}
	\item $A$ is squarefree and odd,
	\item $D=B^2+C^2$ is squarefree and $B,C>0$,
	\item $\gcd(A,D)=1$.
\end{itemize}
Note that $K$ is totally real if $A>0$ and totally imaginary if $A<0$. In the following, there are 5 cases to consider:
\begin{itemize}
	\item[(i)] $D$ even;
	\item[(ii)] $D$ and $B$ odd;
	\item[(iii)] $D$ odd and $B$ even, $A+B\equiv3\mod{4}$;
	\item[(iv)] $D$ odd and $B$ even, $A+B\equiv1\mod{4}$, $A\equiv C\mod{4}$;
	\item[(v)] $D$ odd and $B$ even, $A+B\equiv1\mod{4}$, $A\equiv-C\mod{4}$.
\end{itemize}
Define
\[	\epsilon=	\begin{cases}
				-1	&\text{in case (v)},\\
				1	&\text{otherwise.}
			\end{cases}
\]
Let $\beta=\sqrt{A(D-B\sqrt{D})}$ and let $\sigma$ be the generator of $\Gal(K/\QQ)$ such that $\sigma^{\epsilon}(\alpha)=\beta$. We introduce the following normal basis $(\gamma_0,\gamma_1,\gamma_2,\gamma_3)$ of $K/\QQ$
\begin{align*}
	\gamma_0&=\frac{1}{4}\left(1+\sqrt{D}+\alpha+\epsilon\beta\right)\\
	\gamma_1&=\frac{1}{4}\left(1-\sqrt{D}-\alpha+\epsilon\beta\right)\\
	\gamma_2&=\frac{1}{4}\left(1+\sqrt{D}-\alpha-\epsilon\beta\right)\\
	\gamma_3&=\frac{1}{4}\left(1-\sqrt{D}+\alpha-\epsilon\beta\right),
\end{align*}
so that $\gamma_i=\sigma^i(\gamma_0)$.
One can show that $\disc(\gamma_0,\gamma_1,\gamma_2,\gamma_3)=A^2D^3$. Let $\Gamma$ be the lattice generated by the $\gamma_i$. In cases (iv) and (v), \cite{Spearman-Williams} shows that $(\gamma_0,\gamma_1,\gamma_2,\gamma_3)$ is an \textit{integral} basis of $K$. In the remaining cases, the discriminants \cite{HHRWH} and integral bases \cite{Hudson-Williams} are
\begin{itemize}
	\item[(i)] $\Delta_K=2^8A^2D^3$, basis: $(1,\sqrt{D},\alpha,\beta)$;
	\item[(ii)] $\Delta_K=2^6A^2D^3$, basis: $(1,\frac{1+\sqrt{D}}{2},\alpha,\beta)$;
	\item[(iii)] $\Delta_K=2^4A^2D^3$, basis: $(1,\frac{1+\sqrt{D}}{2},\frac{\alpha+\beta}{2},\frac{\alpha-\beta}{2})$.
\end{itemize}
With a few simple computations, we obtain the following.
\begin{lemma}\label{lem:C4_bases}
	In all cases, $\mc{O}_K$ is a sublattice of $\Gamma$. In cases (i)--(iii), we may take as an integral basis:
	\begin{itemize}
		\item[(i)] $(1,2(\gamma_0+\gamma_2),2(\gamma_0+\gamma_3),2(\gamma_0+\gamma_1))$;
		\item[(ii)] $(1,\gamma_0+\gamma_2,2(\gamma_0+\gamma_3),2(\gamma_0+\gamma_1))$;
		\item[(iii)] $(1,\gamma_0+\gamma_2,\gamma_0-\gamma_2,\gamma_3-\gamma_1)$.
	\end{itemize}
\end{lemma}
\begin{proof}
	We simply note that
	\begin{align*}
		2(\gamma_0+\gamma_2)&=1+\sqrt{D},\\
		2(\gamma_0+\gamma_3)&=1+\alpha,\\
		2(\gamma_0+\gamma_1)&=1+\beta,\\
		\gamma_0-\gamma_2&=\frac{\alpha+\beta}{2},\text{ and}\\
		\gamma_3-\gamma_1&=\frac{\alpha-\beta}{2}.
	\end{align*}
\end{proof}

We will repeatedly use the following simple result whose proof we leave to the reader.
\begin{lemma}\label{lem:C4_trace_is_0}
	The trace of each of $\sqrt{D},\alpha,$ and $\beta$ is zero. The three pairwise inner products of $j(\sqrt{D}),j(\alpha),$ and $j(\beta)$ are all zero. Furthermore,
	\begin{equation}
		\langle j(\alpha),j(\alpha)\rangle=\langle j(\beta),j(\beta)\rangle=4|A|D.
	\end{equation}
\end{lemma}
A simple consequence is the following.
\begin{lemma}\label{lem:C4_Gamma}
	The elements $\gamma_0^\perp,\gamma_1^\perp,\gamma_2^\perp,\gamma_3^\perp$ form an obtuse superbase of $\Gamma^\perp$; indeed,
	\begin{align*}
	\gamma_0^\perp&=\sqrt{D}+\alpha+\epsilon\beta\\
	\gamma_1^\perp&=-\sqrt{D}-\alpha+\epsilon\beta\\
	\gamma_2^\perp&=\sqrt{D}-\alpha-\epsilon\beta\\
	\gamma_3^\perp&=-\sqrt{D}+\alpha-\epsilon\beta.
\end{align*}
Its Gram matrix is
\begin{equation}
	\begin{pmatrix}
		4D(1+2|A|)&-4D&4D(1-2|A|)&-4D\\[5pt]
		-4D&4D(1+2|A|)&-4D&4D(1-2|A|)\\[5pt]
		4D(1-2|A|)&-4D&4D(1+2|A|)&-4D\\[5pt]
		-4D&4D(1-2|A|)&-4D&4D(1+2|A|)
	\end{pmatrix}
\end{equation}
and its conorm diagram is that of Figure~\ref{fig:conorms_tI}(a) with $P_1=4D$ and $P_2=4D(2|A|-1)$.
\end{lemma}

We will use the following lemma to translate between the parameters $A,B,C,D$ and $\Delta_2,\mc{N}$.
\begin{lemma}\label{lem:translateABCDND2}
	Let $p$ be an odd prime.
	\begin{enumerate}
		\item The valuation $v_p(\Delta_K)=3$ if and only if $p$ ramifies in both $K_2$ and $K/K_2$.
		\item The prime $2$ ramifies in $K_2$ if and only if $v_2(\Delta_2)=3$.
		\item When $2$ is unramified in $K$,
		\[
			\left(\frac{|\Delta_2|}{\mc{N}}\right)^{1/4}=|A|^{-1/2}.
		\]
		\item When $2$ ramifies in $K_2$,
		\[
			\left(\frac{4|\Delta_2|}{\mc{N}}\right)^{1/4}=|A|^{-1/2}.
		\]
		\item When $2$ ramifies in $K$, but not in $K_2$,
		\[
			\left(\frac{4|\Delta_2|}{\mc{N}}\right)^{1/4}=\begin{cases}
													(4|A|)^{-1/2}&\text{in case (ii)},\\
													(2|A|)^{-1/2}&\text{in case (iii)}.
												\end{cases}
		\]
	\end{enumerate}
\end{lemma}
\begin{proof}
	The formula for discriminants in a tower implies that
	\begin{equation}\label{eqn:disctowers}
		|\Delta_K|=\mc{N}\Delta_2^2.
	\end{equation}
	If $p$ is odd, $p\,||\,\Delta_2$ if and only if $p$ is ramified. In this case, $p$ only contributes a $p^2$ to $\Delta_K$. Therefore, writing $p\mc{O}_{K_2}=\mf{p}^2$, if $v_p(\Delta_K)=3$, then we must have $\mf{p}\mc{O}_K=\mf{P}^2$. And if this is the case, since $p$ is odd, the ramification of $\mf{p}$ in $K/K_2$ is tame so that $\mf{P}^{2-1}$ exactly divides the different of $K/K_2$. The relative norm of $\mf{P}$ is $\mf{p}$, so $\mf{p}\,||\,\Delta(K/K_2)$. Since the norm of $\mf{p}$ is $p$, we have $p\,||\,\mc{N}$, so that $v_p(\Delta_K)=3$, as claimed.
	
	This implies that if $2$ is unramified in $K$, then $\Delta_2=D$ and $\mc{N}=A^2|D|$, so that $\dfrac{|\Delta_2|}{\mc{N}}=A^{-2}$, as desired.
	
	Let us now consider when $2$ ramifies in $K$. By definition, $2$ ramifies in $K_2$ if and only if $2\mid D$. Since $D$ is a squarefree sum of two squares, $2\nmid D$ if and only if $D\equiv1\mod{4}$. By construction, $K_2=\QQ(\sqrt{D})$. Combining these facts gives that $2$ is unramified in $K_2$ if and only if $\Delta_2=D$. Otherwise, $\Delta_2=4D$ and $2^3\,||\,\Delta_2$. We may now solve for $\mc{N}$ in \eqref{eqn:disctowers}. We obtain
	\[
		\mc{N}=\begin{cases}
					2^4A^2D&\text{in cases (i) and (iii)},\\
					2^6A^2D&\text{in case (ii)}.
				\end{cases}
	\]
	This yields the claimed formulas for $\left(\frac{4|\Delta_2|}{\mc{N}}\right)^{1/4}$.
\end{proof}

\subsection{$K$ unramified at $2$}
When $2$ is unramified in $K$, the elements $\gamma_0^\perp,\gamma_1^\perp,\gamma_2^\perp,\gamma_3^\perp$ form an obtuse superbase of $\mc{O}_K^\perp$. By Lemma~\ref{lem:C4_Gamma}, we see that $j(\mc{O}_K^\perp)$ is a body-centered tetragonal lattice with side lengths $a=4\sqrt{|A|D}$ and $c=4\sqrt{D}$. Thus, $\frac{c}{a}=|A|^{-1/2}\leq1$ with equality exactly when $|A|=1$. Since $\mc{O}_K=\Gamma$, its discriminant is $A^2D^3$, and hence is a cube exactly when $|A|=1$. This completes the proof of part (a) of Theorem~\ref{thm:C4_main_theorem}.

\subsection{$K$ ramified at $2$}
\subsubsection{Case (i): $D$ even}
\begin{proposition}\label{prop:C4_case_i}
	The elements $-4\gamma_0^\perp,2(\gamma_0^\perp+\gamma_1^\perp),2(\gamma_0^\perp+\gamma_2^\perp),2(\gamma_0^\perp+\gamma_3^\perp)$ form an obtuse superbase of $\mc{O}_K^\perp$. Its Gram matrix (scaled by $2^{-6}$) is
	\[
		\begin{pmatrix}
			D(2|A|+1)&-|A|D&-D&-|A|D\\
			-|A|D&|A|D&0&0\\
			-D&0&D&0\\
			-|A|D&0&0&|A|D
		\end{pmatrix}
	\]
	yielding a conorm diagram as in Figure~\ref{fig:conorms_tP} with $a=\sqrt{|A|D}$ and $c=\sqrt{D}$.
\end{proposition}
\begin{proof}
	The superbaseness follows from Lemma~\ref{lem:C4_bases} and the obtuseness from Lemma~\ref{lem:C4_Gamma}. Determining the Gram matrix is a simple computation and the conorm diagram is exactly as stated.
\end{proof}
This shows that $j(\mc{O}_K^\perp)$ is a primitive tetragonal lattice with side lengths $a=\sqrt{|A|D}$ and $c=\sqrt{D}$. Thus, $\frac{c}{a}=|A|^{-1/2}\leq1$ with equality exactly when $|A|=1$. Since the discriminant of $\mc{O}_K$ is $2^{11}A^2\left(\frac{D}{2}\right)^3$ in this case, we have completed the proof of part (b) of Theorem~\ref{thm:C4_main_theorem} in case (i).

It remains to deal with cases (ii) and (iii); in particular, we must show that neither of these cases give cubic lattices.

\subsubsection{Case (ii): $D$ and $B$ odd}
\begin{proposition}\label{prop:C4_case_ii}
	The elements $\gamma_2^\perp-3\gamma_0^\perp,2(\gamma_0^\perp+\gamma_1^\perp),\gamma_0^\perp+\gamma_2^\perp,2(\gamma_0^\perp+\gamma_3^\perp)$ form an obtuse superbase of $\mc{O}_K^\perp$. Its Gram matrix (scaled by $2^{-4}$) is
	\[
		\begin{pmatrix}
			D(8|A|+1)&-4|A|D&-D&-4|A|D\\
			-4|A|D&4|A|D&0&0\\
			-D&0&D&0\\
			-4|A|D&0&0&4|A|D
		\end{pmatrix}
	\]
	yielding a conorm diagram as in Figure~\ref{fig:conorms_tP} with $a=2\sqrt{|A|D}$ and $c=\sqrt{D}$. In particular, $a\neq c$.
\end{proposition}
\begin{proof}
	Again, the superbaseness follows from Lemma~\ref{lem:C4_bases} and the obtuseness from Lemma~\ref{lem:C4_Gamma}. Determining the Gram matrix is again a simple computation and the conorm diagram is exactly as stated. If $a=c$, then $A$ would not be an integer.
\end{proof}
\subsubsection{Case (iii): $D$ odd, $A+B\equiv3\mod{4}$}
\begin{proposition}\label{prop:C4_case_iii}
	The elements $\gamma_1^\perp-\gamma_3^\perp-2\gamma_0^\perp,\gamma_3^\perp-\gamma_1^\perp,\gamma_0^\perp+\gamma_2^\perp,\gamma_0^\perp-\gamma_2^\perp$ form an obtuse superbase of $\mc{O}_K^\perp$. Its Gram matrix (scaled by $2^{-4}$) is
	\[
		\begin{pmatrix}
			D(4|A|+1)&-2|A|D&-D&-2|A|D\\
			-2|A|D&2|A|D&0&0\\
			-D&0&D&0\\
			-2|A|D&0&0&2|A|D
		\end{pmatrix}
	\]
	yielding a conorm diagram as in Figure~\ref{fig:conorms_tP} with $a=\sqrt{2|A|D}$ and $c=\sqrt{D}$. In particular, $a\neq c$.
\end{proposition}
\begin{proof}
	The proof is the same as the previous case.
\end{proof}
This ends the proof of Theorem~\ref{thm:C4_main_theorem}

\section{The distribution of shapes of $C_4$-quartic fields}\label{sec:C4shapes_distribution}
\newcommand{\s}{\Sigma}

The shapes of $C_4$-fields form a discrete set of points that has no accumulation point in the space of shapes and as such they cannot be equidistributed in some (positive-dimensional) submanifold of the space of shapes. In this section, we will therefore determine asymptotics for the set of $C_4$-fields of given shape (and signature). Given the description we have of $C_4$-fields from \S\ref{sec:C4fields_description}, this question reduces to certain asymptotics of well-known arithmetic functions that we now describe.

\subsection{Some notation}\label{sec:C4_dist_notation}

For $\s$ a set of prime numbers and $n\in\ZZ$, we write $(n,\s)=1$ to mean that $n$ is relatively prime to every element of $\s$. Let $\s_0$ be the set of primes congruent to $2$ or $3$ modulo $4$ and let $\s=\s_0\sqcup\s_1$, where $\s_1$ is a finite set of primes disjoint from $\s_0$. For a non-zero integer $A$, let $\s_A=\s_0\cup\{p\mid A\}$.

Given an arithmetic function $f:\ZZ_{\geq1}\rightarrow\CC$, we let
\[
	L(s,f):=\sum_{n\geq1}\frac{f(n)}{n^s}.
\]

For $n\in\ZZ_{\geq1}$, let $\omega(n)$ denote the number of distinct prime divisors of $n$ and let $\mu(n)$ denote the M\"obius $\mu$-function. For one of the subsets $U=\{1\},\{5\}$, or $\{1,5\}$ of $(\ZZ/8\ZZ)^\times$, let
\[
	f_{\s,U}(n):=\begin{cases}
				|\mu(n)|2^{\omega(n)}&(n,\s)=1\text{ and }n\mod{8}\in U\\
				0&\text{otherwise}.
			\end{cases}
\]
(we will allow ourselves to drop the subscript $U$ when $U=\{1,5\}$; indeed, in this case, $U$ does not impose any extra condition since $\Sigma$ contains all primes that are not $1\mod{4}$). In the next section, we will reduce the determination of the asymptotics for $C_4$-quartic fields of a given shape to that of the following functions:
\[
	F_{\s,U}(Y)=\sum_{1\leq n\leq Y}f_{\s,U}(n).
\]

\subsection{Reduction to asymptotics of simpler arithmetic functions}
Recall from \S\ref{sec:C4fields_description} that a $C_4$-quartic field is given uniquely by $K=\QQ(\alpha)$, where $\alpha=\sqrt{A(D+B\sqrt{D})}$ with $A,B,C,D\in\ZZ$ satisfying
\begin{itemize}
	\item $A$ is squarefree and odd,
	\item $D=B^2+C^2$ is squarefree and $B,C>0$,
	\item $\gcd(A,D)=1$.
\end{itemize}
Then, Theorem~\ref{thm:C4_main_theorem} tells use that the shape of $K$ depends only on $A$ and whether the field is in case (i), (ii), (iii), or the unramified-at-$2$ cases (iv) and (v) (which we will combine here). We denote the corresponding lattice shapes by $\Lambda_A^{(\mathrm{i})},\Lambda_A^{(\mathrm{ii})},\Lambda_A^{(\mathrm{iii})}$, and $\Lambda_A^{(\nr)}$, respectively (with $\nr$ denoting the unramified cases). For each $?\in\{(\mathrm{i}),(\mathrm{ii}),(\mathrm{iii}),\nr\}$, let
\[
	N^{?,\pm}_A(X)=\#\{K\text{ a }C_4\text{-quartic field such that }|\Delta_K|\leq X,\sgn(A)=\pm1,\sh(K)=\Lambda^?_A\}.
\]

\begin{theorem}\label{thm:C4_asymptotics}
	Let  $A$ be squarefree and odd. Let
	\[C_{\s}:=\prod_p\left(1+\frac{f_{\s}(p)}{p}\right)\left(1-\frac{1}{p}\right),
	\]
	so that
	\[C_{\s_A}:=C_{\s_0}\cdot\left(\prod_{\substack{p\mid A\\p\equiv1\mod{4}}}\frac{p}{p+2}\right).
	\]
	Then, for all $\epsilon>0$,
	\begin{align}
		N^{(\mathrm{i}),\pm}_A(X)&=\frac{C_{\s_A}}{\left(2^{14}A^2\right)^{1/3}}X^{1/3}+O(X^{1/3}/(\log X)^{1-\epsilon}),\\
		N^{(\mathrm{ii}),\pm}_A(X)&=\frac{C_{\s_A}}{\left(2^{9}A^2\right)^{1/3}}X^{1/3}+O(X^{1/3}/(\log X)^{1-\epsilon}),\\
		N^{(\mathrm{iii}),\pm}_A(X)&=\frac{C_{\s_A}}{\left(2^{10}A^2\right)^{1/3}}X^{1/3}+o(X^{1/3}),\\
		N^{(\nr),\pm}_A(X)&=\frac{C_{\s_A}}{\left(2^{6}A^2\right)^{1/3}}X^{1/3}+o(X^{1/3}).
	\end{align}
\end{theorem}

\begin{remark}
	As alluded to in the introduction, these counts present arithmetic behaviour that is not compatible with being well-behaved with respect to a measure inherited by a $G$-action for any (non-trivial) subgroup $G$ of $\gl_n(\RR)$. Indeed, within a given case, if, e.g., $A_2=pA_1$, for some prime $p\nmid A_1$, the proportion of fields with shape given by $A_1$ versus shape given by $A_2$ is
	\[
		\begin{cases}
			p^{2/3}+\dfrac{2}{p^{1/3}}	& p\equiv1\mod{4}\\
			p^{2/3}				& p\equiv3\mod{4}.
		\end{cases}
	\]
	Since the parameter $A$ is a scaling parameter, an invariant measure would require that this proportion not depend on the congruence class of $p$ modulo $4$.
\end{remark}

We reduce the asymptotics of these functions to asymptotics of the functions $F_{\s,U}(Y)$ of the previous section case-by-case. First, let us note that basically what we are trying to do comes down to counting how many ways a given $D$ can be written as a sum of two squares. Let us briefly recall what is known about this. Let $Q(D)$ denote the number of ways of writing $D$ as a sum of two squares $B^2+C^2$ without regard to the order or signs of $B$ and $C$. It has been known for quite some time that when $D$ is squarefree and not divisible by any primes that are $3\mod{4}$, we have that
\[
	Q(D)=\begin{cases}
			2^{\omega(D)-2}&D\text{ even}\\
			2^{\omega(D)-1}&D\text{ odd}.
		\end{cases}
\]
For the cases (iii) and $\nr$, we will need the following lemma.
\begin{lemma}\label{lem:QD_with_cong}
	Let $D\in\ZZ_{\geq1}$ be odd. Then, $D$ can be written as $B^2+C^2$ with ${B\equiv 0\mod{4}}$ if and only if $D\equiv1\mod{8}$. In particular, if $D$ can be written in this way, then all ways of writing $D$ as a sum of two squares have $B$ (or $C$) $\equiv0\mod{4}$.
\end{lemma}
\begin{proof}
	If $D=B^2+C^2$ with $B\equiv0\mod{4}$, then $D\equiv1\mod{8}$ since $1$ is the only odd square modulo $8$. Conversely, since $D$ is odd, exactly one of $B$ or $C$ is even, so that if $D$ can't be written as $B^2+C^2$ with $B\equiv0\mod{4}$, then it must be that $D$ can be written as $B^2+C^2$ with $B\equiv2\mod{4}$ and $C$ odd. But then $B=2B^\prime$, with $B^\prime$ odd and
	\begin{align*}
		D&=(2B^\prime)^2+C^2\\
		&\equiv4\cdot1+1\mod{8}\\
		&\equiv5\mod{8}.
	\end{align*}
\end{proof}

We now proceed case-by-case to relate asymptotics of $N^{?,\pm}_A(X)$ to those of $F_{\s_A}(Y)$.

Case (i): $D$ even. We have that
\begin{align}
	N^{(\mathrm{i}),\pm}_A(X)&=\sum_{\substack{2\leq D\leq\left(\frac{X}{2^8A^2}\right)^{1/3}\\ D\text{ squarefree}\\ (D/2,\s_A)=1\\D\text{ even} }} Q(D)\\
		&=\sum_{\substack{1\leq D^\prime\leq\frac{1}{2}\left(\frac{X}{2^8A^2}\right)^{1/3}\\ D^\prime\text{ squarefree}\\ (D^\prime,\s_A)=1\\D^\prime\text{ odd} }} 2^{\omega(D^\prime)-1}\\
		&=\frac{1}{2}F_{\s_A}\!\left(\frac{1}{2}\left(\frac{X}{2^8A^2}\right)^{1/3}\right).
\end{align}

Case (ii): $D,B$ odd. Note that if $D=B^2+C^2$ is odd, then exactly one of $B$ or $C$ is odd, so that we are again counting the appropriate $D$ with multiplicity $Q(D)$, i.e.
\begin{align}
	N^{(\mathrm{ii}),\pm}_A(X)&=\sum_{\substack{1\leq D\leq\left(\frac{X}{2^6A^2}\right)^{1/3}\\ D\text{ squarefree}\\ (D,\s_A)=1 }} Q(D)\\
		&=\sum_{\substack{1\leq D\leq\left(\frac{X}{2^6A^2}\right)^{1/3}\\ D\text{ squarefree}\\ (D,\s_A)=1 }} 2^{\omega(D)-1}\\
		&=\frac{1}{2}F_{\s_A}\!\left(\left(\frac{X}{2^6A^2}\right)^{1/3}\right).
\end{align}

Case (iii): $D$ odd, $A+B\equiv3\mod{4}$. In this case, $B$ is required not only to be even, but to satisfy a congruence condition modulo $4$. If $A\equiv3\mod{4}$, then $B\equiv0\mod{4}$, and if $A\equiv1\mod{4}$, then $B\equiv2\mod{4}$. In view of Lemma~\ref{lem:QD_with_cong}, let $U=\{1\}$ or $\{5\}$ according to whether $A$ is $3$ or $1\mod{4}$, so that for a given choice of $A$, we must count those $D$ whose congruence class module 8 is in $U$, i.e
\begin{align}
	N^{(\mathrm{iii}),\pm}_A(X)&=\sum_{\substack{1\leq D\leq\left(\frac{X}{2^4A^2}\right)^{1/3}\\ D\text{ squarefree}\\ (D,\s_A)=1\\D\mod{8}\,\in\, U }} Q(D)\\
		&=\sum_{\substack{1\leq D\leq\left(\frac{X}{2^4A^2}\right)^{1/3}\\ D\text{ squarefree}\\ (D,\s_A)=1\\D\mod{8}\,\in\, U }} 2^{\omega(D)-1}\\
		&=\frac{1}{2}F_{\s_A,U}\!\left(\left(\frac{X}{2^4A^2}\right)^{1/3}\right).
\end{align}

Case $\nr$: $2$ unramified in $K$. In this case, $D$ is odd and $A+B\equiv1\mod{4}$. As such, the answer is along the same lines as in case (iii), but with the opposite choice of $U$, i.e.\ $U=\{1\}$ or $\{5\}$ according to whether $A$ is $1$ or $3\mod{4}$. Then,
\begin{align}
	N^{(\mathrm{\nr}),\pm}_A(X)&=\sum_{\substack{1\leq D\leq\left(\frac{X}{A^2}\right)^{1/3}\\ D\text{ squarefree}\\ (D,\s_A)=1\\D\mod{8}\,\in\, U }} Q(D)\\
		&=\sum_{\substack{1\leq D\leq\left(\frac{X}{A^2}\right)^{1/3}\\ D\text{ squarefree}\\ (D,\s_A)=1\\D\mod{8}\,\in\, U }} 2^{\omega(D)-1}\\
		&=\frac{1}{2}F_{\s_A,U}\!\left(\left(\frac{X}{A^2}\right)^{1/3}\right).
\end{align}

\subsection{Asymptotics of some arithmetic functions}
In this section, we will use the Wirsing--Odoni method and Wiener--Ikehara Tauberian Theorem to obtain asymptotics for $F_{\s,U}(Y)$. We will need the following lemma on the absolute convergence of Dirichlet series of certain multiplicative functions.
\begin{lemma}\label{lem:convergence}
	Suppose $h(n)$ is a multiplicative function satisfying the following three properties for all primes $p$ and all positive integers $k$:
	\begin{enumerate}
		\item there is an $M\geq1$ such that $|h(p^k)|\leq M$;
		\item there is an integer $K>0$ such that $h(p^k)=0$ for all $k>K$;
		\item $h(p)=0$.
	\end{enumerate}
	Then, the Dirichlet series
	\[
		H(s)=\sum_{n\geq1}\frac{h(n)}{n^s}
	\]
	converges absolutely for $\Re(s)>\frac{K-1}{K}$.
\end{lemma}
\begin{proof}	
	Given
	\[
		n=\prod_{i=1}^rp_i^{e_i}
	\] for distinct primes $p_i$ and $e_i\in\ZZ_{\geq2}$, let
	\[
		n^\prime=\prod_{i=1}^rp_i^{e_i-1}>1.
	\]
	This gives a bijection between squarefull $n$ and integers $n^\prime>1$. We have that $|h(n)|,|h(n^\prime)|\leq M^r$. Let $\epsilon>0$. For sufficiently large $n^\prime$,
	\[r\leq \log_{M}({n^\prime}^{\epsilon/2}),
	\]
	so that
	\[
		|h(n^\prime)|=O({n^\prime}^{\epsilon/2}).
	\]
	If $n$ is such that $e_i\leq K$ for all $i$, then
	\[
		(e_i-1)\leq e_i\frac{K-1}{K},
	\]
	so that if $\sigma=\frac{K-1}{K}+\epsilon$, then
	\[
		n^\sigma=\prod_{i=1}^rp_i^{e_i\sigma}=\prod_{i=1}^rp_i^{e_i(K-1)/K+e_i\epsilon}\geq \prod_{i=1}^rp_i^{e_i-1+(e_i-1)\epsilon}={n^{\prime}}^{1+\epsilon}
	\]
	We therefore obtain
	\[
		\sum_{n\geq1}\frac{|h(n)|}{n^\sigma}\leq\sum_{n^\prime\geq1}\frac{|h(n^\prime)|}{{n^\prime}^{1+\epsilon}}=O\left(\sum_{n^\prime\geq1}\frac{1}{{n^\prime}^{1+\epsilon/2}}\right)<\infty.
	\]

\end{proof}
Throughout this section, for $j=3,5,7$, we let $\chi_j$ denote the (unique) Dirichlet character modulo $8$ whose kernel is generated by $j$ mod $8$. We heartily thank Robert Lemke Oliver for pointing us to the following wonderfully simple approach using the Wirsing--Odoni method!

\begin{proposition}
	With the notation of \S\ref{sec:C4_dist_notation}, we have that, for all $\epsilon>0$,
	\begin{equation}
		F_{\s}(Y)=C_\s Y+O(Y/(\log Y)^{1-\epsilon})
	\end{equation}
	where
	\begin{equation}
		C_\s=\prod_p\left(1+\frac{f_\s(p)}{p}\right)\left(1-\frac{1}{p}\right).
	\end{equation}
\end{proposition}
\begin{proof}
	We use the Wirsing--Odoni method as laid out in \cite[Proposition~4]{FMS}.
	The first stipulation of the Wirsing--Odoni method is that it applies to multiplicative functions of which $f_\s(n)$ is an example. Next, since
	\[
		0\leq f_{\s}(p^r)\leq2
	\]
	for all prime powers $p^r$, we may take $u=2$ and $v=0$ in \cite[Proposition~4]{FMS}. Finally, we must find real numbers $\xi>0$ and $0<\beta<1$ such that
	\[
		\sum_{p<X}f_{\s}(p)=\xi\frac{X}{\log X}+O\!\left(\frac{X}{(\log X)^{1+\beta}}\right).
	\]
	The left-hand side is simply $2$ times the sum of all primes less than $X$ that are not in $\s$. The condition of not being in $\s$ only excludes finitely many primes beyond the congruence condition of being $1$ modulo $4$, so that all we need is Dirichlet's theorem on primes in arithmetic progressions, as well as the Siegel--Walfisz Theorem (see e.g.\ \cite[Corollary~5.29]{Iwaniec-Kowalski}) for the error term, to conclude that we may take any $\beta\in(0,1)$ and $\xi=1$. Plugging these numbers into the conclusion of \cite[Proposition~4]{FMS} yields the stated result.
\end{proof}

To deal with $F_{\s,\{1\}}(Y)$ and $F_{\s,\{5\}}(Y)$, we will use the Wiener--Ikehara Tauberian Theorem as in \cite[Exercise~3.3.3]{Murty}.
\begin{proposition}
	If $U=\{1\}$ or $\{5\}$, then
	\begin{equation}
		F_{\s,U}(Y)=\frac{1}{2}C_\s Y+o(Y)
	\end{equation}
\end{proposition}
\begin{proof}
	For $U=\{1\}$ or $\{5\}$, $f_{\s,U}(n)$ is not multiplicative so the Wirsing--Odoni method does not apply. Since
	\[
		F_\s(Y)=F_{\s,\{1\}}(Y)+F_{\s,\{5\}}(Y),
	\]
	it suffices to prove the result for $F_{\s,\{1\}}(Y)$. By elementary mathematics (or the orthogonality of Dirichlet characters, if you're not into that whole brevity thing), we have that
	\begin{equation}
		F_{\s,\{1\}}(Y)=\frac{1}{2}\left(F_\s(Y)+\sum_{1\leq n\leq Y}\chi_3(n)f_\s(n)\right).
	\end{equation}
	We therefore concentrate on the Dirichlet series $L(s,f_{\s,\chi_3})$, where $f_{\s,\chi_3}(n)=\chi_3(n)f_\s(n)$. It suffices to show that
	\[ \sum_{1\leq n\leq Y}\chi_3(n)f_\s(n)=o(Y).
	\]
	Since $|f_{\s,\chi_3}(n)|\leq f_\s(n)$, we first study $L(s,f_\s)$.
	
	Let
	\[
		H_{\s_0}(s):=L(s,f_{\s_0})\big(\zeta(s)L(s,\chi_5)\big)^{-1}=\sum_{n\geq1}\frac{h_{\s_0}(n)}{n^s},
	\]
	where $h_{\s_0}=f_{\s_0}\ast\mu\ast\mu\chi_5$, where $\ast$ denotes Dirichlet convolution. Then, $h_{\s_0}(n)$ is multiplicative and
	\[
		h_{\s_0}(p^k)=\begin{cases}
					1&\text{if }k=0\\
					-3&\text{if }k=2\text{ and }p\equiv1\mod{4}\\
					-1&\text{if }k=2\text{ and }p\not\equiv1\mod{4}\\
					2&\text{if }k=3\text{ and }p\equiv1\mod{4}\\
					0&\text{otherwise.}
				\end{cases}
	\]
	Therefore, by the previous lemma, the Dirichlet series for $H_{\s_0}(s)$ converges absolutely for $\Re(s)>2/3$ and defines $H_{\s_0}(s)$ as a non-zero analytic function in that region. Letting
	\[
		H_\s(s):=H_{\s_0}(s)\cdot\prod_{p\in\s\setminus\s_0}\left(1+2p^{-s}\right)^{-1},
	\]
	we obtain the factorization
	\begin{equation}\label{eqn:Lsfs_factorization}
		L(s,f_\s)=\zeta(s)L(s,\chi_5)H_\s(s).
	\end{equation}
	Since $\zeta(s)$ and $L(s,\chi_5)$ (and $H_\s(s)$) converge absolutely for $\Re(s)>1$, so does $L(s,f_\s)$. Furthermore, $L(s,\chi_5)$ has analytic continuation to the entire complex plane and $L(1,\chi_5)\neq0$, so that $L(s,f_\s)$ has meromorphic continuation to $\Re(s)>2/3$ with a simple pole at $s=1$. In order to apply the Wiener--Ikehara Tauberian theorem to $f_{\s,\chi_3}$ and obtain our result, it now suffices to show that $L(s,f_{\s,\chi_3})$ converges absolutely for $\Re(s)>1$ and extends to an analytic function on $\Re(s)\geq1$.
	
	We proceed along the same lines as the previous paragraph, letting
	\[
		H_{\s_0,\chi_3}(s):=L(s,f_{\s_0,\chi_3})\big(L(s,\chi_3)L(s,\chi_7)\big)^{-1}=\sum_{n\geq1}\frac{h_{\s_0,\chi_3}(n)}{n^s},
	\]
	where $h_{\s_0,\chi_3}=f_{\s_0,\chi_3}\ast\mu\chi_3\ast\mu\chi_7$ and is again a multiplicative function. We have that
	\[
		h_{\s_0,\chi_3}(p^k)=\begin{cases}
					1&\text{if }k=0\\
					-3&\text{if }k=2\text{ and }p\equiv1\mod{4}\\
					-1&\text{if }k=2\text{ and }p\not\equiv1\mod{4}\\
					2&\text{if }k=3\text{ and }p\equiv1\mod{8}\\
					-2&\text{if }k=3\text{ and }p\equiv5\mod{8}\\
					0&\text{otherwise.}
				\end{cases}
	\]
	This implies that $H_{\s_0,\chi_3}(s)$ converges absolutely for $\Re(s)>2/3$, and similarly for
	\[
		H_{\s,\chi_3}(s):=H_{\s_0,\chi_3}(s)\cdot\prod_{p\in\s\setminus\s_0}\left(1+2\chi_3(p)p^{-s}\right)^{-1}.
	\]
	Since $L(s,\chi_3)$ and $L(s,\chi_5)$ both converge absolutely for $\Re(s)>1$ and extend to analytic functions on the entire complex plane, $L(s,f_{\s,\chi_3})$ converges absolutely for $\Re(s)>1$ and extends to an analytic function on $\Re(s)>2/3$. The Wiener--Ikehara Tauberian Theorem applies to yield the result.
\end{proof}

\subsection*{Acknowledgments}
We would like to thank Jamal Hassan, Erik Holmes, Robert Lemke Oliver, Jacob Tsimerman, Ila Varma, and Melanie Matchett Wood for some helpful conversations.

\bibliographystyle{amsalpha}
\bibliography{quartic}

\newcommand{\etalchar}[1]{$^{#1}$}
\providecommand{\bysame}{\leavevmode\hbox to3em{\hrulefill}\thinspace}
\providecommand{\MR}{\relax\ifhmode\unskip\space\fi MR }
\providecommand{\MRhref}[2]{%
  \href{http://www.ams.org/mathscinet-getitem?mr=#1}{#2}
}
\providecommand{\href}[2]{#2}
\begin{thebibliography}{HHR{\etalchar{+}}86}

\bibitem[Bai80]{Baily}
Andrew~Marc Baily, \emph{On the density of discriminants of quartic fields}, J.
  Reine Angew. Math. \textbf{315} (1980), 190--210. \MR{564533}

\bibitem[BH16]{Manjul-Piper}
Manjul Bhargava and Piper H, \emph{The equidistribution of lattice shapes of
  rings of integers in cubic, quartic, and quintic number fields}, Compositio
  Mathematica \textbf{152} (2016), no.~6, 1111--1120. \MR{3518306}

\bibitem[Bha05]{ManulQuartic}
Manjul Bhargava, \emph{The density of discriminants of quartic rings and
  fields}, Ann. of Math. (2) \textbf{162} (2005), no.~2, 1031--1063.
  \MR{2183288}

\bibitem[BMS19]{BMS}
Wilmar Bola{\~{n}}os and Guillermo Mantilla-Soler, \emph{The trace form over
  cyclic number fields}, 2019, preprint, available at
  \href{{https://arxiv.org/abs/1904.10080v2}}{arXiv:1904.10080v2 [math.NT]}.

\bibitem[CS92]{Conway-Sloane}
J.~H. Conway and N.~J.~A. Sloane, \emph{Low-dimensional lattices. {VI}.
  {V}orono\u{\i} reduction of three-dimensional lattices}, Proc. Roy. Soc.
  London Ser. A \textbf{436} (1992), no.~1896, 55--68. \MR{1177121}

\bibitem[DH71]{DH}
H.~Davenport and H.~Heilbronn, \emph{On the density of discriminants of cubic
  fields. {II}}, Proc. Roy. Soc. London Ser. A \textbf{322} (1971), no.~1551,
  405--420. \MR{0491593}

\bibitem[Fed91]{Fedorov91}
E.~S. Fedorov, \emph{The symmetry of regular systems of figures}, Zapiski
  Mineralogi\v{c}eskih Ob\v{s}\v{c}estva (2) \textbf{28} (1891), 1--146,
  English translation in \textit{Symmetry of crystals}, ACA Monograph no.\ 7,
  pp.~50--131, New York (1971).

\bibitem[Fed53]{Fedorov85}
\bysame, \emph{Na\v{c}ala u\v{c}eniya o figurah [{E}lements of the study of
  figures]}, Izdat. Akad. Nauk SSSR, Moscow, [1885] 1953, Orig.\ published in
  Zapiski Mineralogi\v{c}eskih Ob\v{s}\v{c}estva (2) \textbf{21} (1885),
  1--279. \MR{0062061}

\bibitem[FMS10]{FMS}
Steven Finch, Greg Martin, and Pascal Sebah, \emph{Roots of unity and nullity
  modulo {$n$}}, Proc. Amer. Math. Soc. \textbf{138} (2010), no.~8, 2729--2743.
  \MR{2644888}

\bibitem[H16]{PiperThesis}
Piper H, \emph{The equidistribution of lattice shapes of rings of integers of
  cubic, quartic, and quintic number fields: an artist's rendering}, Ph.D.
  thesis, Princeton University, 2016, p.~130.

\bibitem[Har17]{PureCubicShapes}
Robert Harron, \emph{The shapes of pure cubic fields}, Proc. Amer. Math. Soc.
  \textbf{145} (2017), no.~2, 509--524. \MR{3577857}

\bibitem[Har19]{ComplexCubics}
Robert Harron, \emph{Equidistribution of shapes of complex cubic fields of
  fixed quadratic resolvent}, 2019, preprint, available at
  \href{{https://arxiv.org/abs/1907.07209}}{arXiv:11907.07209 [math.NT]}.

\bibitem[HHR{\etalchar{+}}86]{HHRWH}
Kenneth Hardy, R.~H. Hudson, D.~Richman, Kenneth~S. Williams, and N.~M. Holtz,
  \emph{Calculation of the class numbers of imaginary cyclic quartic fields},
  Carleton--Ottawa Mathematical Lecture Note Series, vol.~7, 1986.

\bibitem[HW90]{Hudson-Williams}
R.~H. Hudson and K.~S. Williams, \emph{The integers of a cyclic quartic field},
  Rocky Mountain J. Math. \textbf{20} (1990), no.~1, 145--150. \MR{1057983}

\bibitem[IK04]{Iwaniec-Kowalski}
Henryk Iwaniec and Emmanuel Kowalski, \emph{Analytic number theory}, American
  Mathematical Society Colloquium Publications, vol.~53, American Mathematical
  Society, Providence, RI, 2004. \MR{2061214}

\bibitem[Mur08]{Murty}
M.~Ram Murty, \emph{Problems in analytic number theory}, second ed., Graduate
  Texts in Mathematics, vol. 206, Springer, New York, 2008, Readings in
  Mathematics. \MR{2376618}

\bibitem[Neu99]{Neukirch}
J{\"u}rgen Neukirch, \emph{Algebraic number theory}, Grundlehren der
  Mathematischen Wissenschaften, vol. 322, Springer-Verlag, Berlin, 1999,
  Translated from the 1992 German original by Norbert Schappacher. \MR{1697859
  (2000m:11104)}

\bibitem[RGMS19]{RGMS}
Carlos Rivera-Guaca and Guillermo Mantilla-Soler, \emph{A proof of a conjecture
  on trace-zero forms and shapes of number fields}, 2019, preprint, available
  at \href{{https://arxiv.org/abs/1907.09134v2}}{arXiv:1907.09134v2 [math.NT]}.

\bibitem[SW06]{Spearman-Williams}
Blair~K. Spearman and Kenneth~S. Williams, \emph{Cyclic quartic fields with a
  unique normal integral basis}, Far East J. Math. Sci. (FJMS) \textbf{21}
  (2006), no.~2, 235--240. \MR{2247713}

\bibitem[Ter97]{Terr}
David~C. Terr, \emph{The distribution of shapes of cubic orders}, Ph.D. thesis,
  University of California, Berkeley, 1997, p.~137. \MR{2697241}

\bibitem[Wil70]{WilliamsBiquad}
Kenneth~S. Williams, \emph{Integers of biquadratic fields}, Canad. Math. Bull.
  \textbf{13} (1970), 519--526. \MR{0279069}

\end{thebibliography}

\end{document}